\documentclass[12pt]{amsart}
\usepackage{hyperref}
\usepackage{amsfonts}
\usepackage{bbm}
\usepackage{esint}
\usepackage{mathrsfs}
\usepackage{amssymb,amsmath,amsfonts,amsthm,color}
\usepackage{setspace}
\numberwithin{equation}{section}

\DeclareMathOperator{\supp}{supp}

\allowdisplaybreaks

\newtheorem{thm}{\hskip\parindent {Theorem}}[section]
\newtheorem{lem}{\hskip\parindent {Lemma}}[section]
\newtheorem{lemma}{\hskip\parindent {Lemma}}[section]

\newtheorem{pro}{\hskip\parindent{Proposition}}[section]
\newtheorem{prop}{\hskip\parindent{Proposition}}[section]

\theoremstyle{definition}
\newtheorem{definition}{\hspace\parindent{Definition}}[section]

\allowdisplaybreaks

\newcommand{\Z}{\mathbb{Z}}

\newcommand{\R}{\mathbb{R}}

\def\R{\mathbb R}

\def\12{\frac{1}{2}}
\def\M{\mathbb{M}}

\def\binfty {\dot{B}_\infty^{\alpha,\infty}}
\def\b1 {\dot{B}_1^{-\alpha,1}}

\begin{document}

\setcounter{tocdepth}{3} \allowdisplaybreaks

\title[Singular integrals with reflection groups]
{Singular integrals associated with reflection groups on Euclidean space}

\author[Y. Han, J. Li, C. Tan, Z. Wang, X. Wu ]{Yongsheng Han, Ji Li, Chaoqiang Tan, Zipeng Wang, Xinfeng Wu }



\begin{abstract}
    In the field of harmonic analysis, geometric considerations are frequently crucial. Specifically, group actions such as translations, dilations, and rotations on Euclidean space are instrumental. The objective of this paper is to extend the study of singular integrals to include the effects of group reflections on Euclidean space and to establish a $T1$ theorem for these singular integrals.
\end{abstract}

\maketitle
\tableofcontents
\section{Introduction and statement of main results}

It is well established that group structures play a pivotal role in harmonic analysis, with actions such as translations, dilations, and rotations on Euclidean space serving as foundational elements in the study of singular integrals, as referenced by prior works \cite{CZ, CW1, CW2, FR, KS, KV, MC, NRW, NW, PS, SO, S1, S2, S3, S4, S5, SW, ST}. Our primary objective is to expand upon this existing framework by developing the theory of singular integrals under the influence of group reflections on Euclidean spaces.

To elaborate, consider the standard inner product in $\mathbb{R}^{N}$ defined as $\langle x,y\rangle=\sum\limits_{j=1}^Nx_jy_j$ and the associated norm $|x|=\left(\sum\limits_{j=1}^N|x_j|^2\right)^\frac{1}{2}$. For a non-zero vector $\alpha\in \mathbb{R}^{N}$, the reflection map $\sigma_\alpha$ relative to the hyperplane $\alpha^{\perp}$ orthogonal to $\alpha$ is given by:
$$\sigma_\alpha(x)=x-2\frac{\langle x,\alpha\rangle}{|\alpha|^2}\alpha.$$

We will focus on a normalized root system $R$, which constitutes a
finite subset of $\R^N\setminus \{0\}$, adhering to specific
criteria such as the intersection $R\cap \R\alpha=\{\pm \alpha\},$
and the invariance under reflections $\sigma_\alpha R = R$ for every
$\alpha \in R$, with the stipulation that $|\alpha| = \sqrt{2}$. The
set of reflections $\{\sigma_{\alpha}: \alpha\in R\}$ generates a
finite group $G$, known as the finite reflection group associated
with the root system $R$. Throughout this work, we shall fix a normalized root system $R$ associated with reflection group $G$.

Consider the set  $B(x,r):=\{y\in
\mathbb{R}^{N}:|x-y|<r\}$, which represents the Euclidean ball centered at $x \in \mathbb{R}^{N}$ with a radius $r > 0$. Let us define the distance function
$d(x,y):=\min\limits_{\sigma\in G}|x-\sigma(y)|$
between two $G$-orbits $\mathcal{O}(x)$ and $\mathcal{O}(y)$. It is evident that $d(x,y)$ is non-negative and symmetric, meaning $0 \leq d(x,y) = d(y,x) \leq |x - y|$, and satisfies the triangle inequality such that $d(x,y) \leq d(x,z) + d(z,y)$. Moreover, $d(x,y)$ vanishes when $y$ is a reflection of $x$, that is, $y = \sigma(x)$ for each $\sigma \in G$.

For any point $ x\in\mathbb{R}^{N}$ and a given radius $r > 0$, we introduce the set
$\mathcal{O}_B(x,r):=\{y\in \mathbb{R}^{N}:d(x,y)<r\},$
which is noteworthy for its invariance under the reflection group actions, as it holds that $\mathcal{O}_B(x,r)=\mathcal{O}_B(\sigma(x), r)$  for every $\sigma\in G.$

To begin with the theory of singular integrals associated with the reflection group $G$, we commence with an illustrative example.
Let $\phi(x)$ be a Schwartz function in $\mathbb{R}$ and define
$$K(x,y)=\sum\limits_k 2^{kN} \phi(2^kd(x,y)).$$

We observe that $K(x,y)$ remains invariant under the reflection group actions; that is, $K(x,y)=K(\sigma(x),
y)=K(x,\sigma(y))=K(\sigma(x), \sigma(y)),$ for all $\sigma \in G.$

By applying size and smoothness estimates for $2^{kN} \phi(2^kd(x,y))$ and inequalities $|d(x,y)-d(x',y)|\le d(x,x')$ and $|d(x,y)-d(x,y')|\le d(y,y')$, we deduce that for $d(x,y)\not=0,$
$$|K(x,y)|\lesssim \frac{1}{(d(x,y))^N},$$
$$|K(x,y)-K(x',y)|\lesssim \frac{d(x,x')}{(d(x,y))^{N+1}},\qquad \text{if }d(x,x')\leq \frac{1}{2}d(x,y),$$
and
$$|K(x,y)-K(x,y')|\lesssim \frac{d(y,y')}{(d(x,y))^{N+1}},\qquad \text{if }d(y,y')\leq \frac{1}{2}d(x,y).$$

It is important to note that each function $K(x,y)$ fulfilling the aforementioned smoothness conditions is invariant under the reflection group actions as previously stated. In fact, by taking $x'=\sigma(x)$ for each $\sigma\in G$ in the aforementioned smoothness condition, we have that for $d(x,y)\not=0,$
$$|K(x,y)-K(\sigma(x),y)|\lesssim \frac{d(x,\sigma(x))}{(d(x,y))^{N+1}}=0$$
since $d(x, \sigma(x))=0$, which implies $K(x,y)=K(\sigma(x),y).$  Similarly, for $d(x,y)\not=0, K(x,y)=K(x,\sigma(y)).$

To describe the theory of singular integrals associated with reflection group $G$, we first define the $G$-invariant H\"older space. Let $C^\eta_G(\R^N)$,  $\eta > 0$, denote the space of continuous functions $f$ that satisfy
$$\|f\|_{G,\eta}:=\sup\limits_{d(x,y)\ne 0} \frac{|f(x)-f(y)|}{d(x, y)^{\eta}}<\infty $$
and let $C^\eta_{G,,0}(\R^N)$ consist of those functions $f\in C^\eta_G(\R^N)$ with compact support.
It is crucial to note that if $f\in C^\eta_{G}(\R^N)$, then $f(x)=f(\sigma(x))$ for all $\sigma\in G.$

\begin{definition}\label{def1.1}
Let $T$ be an operator initially defined as a mapping from $C_{G,0}^\eta(\R^N)$ to $(C_{G,0}^\eta(\R^N)'$ for $\eta > 0$.
We say that $T$ is a singular integral operator if its kernel $K(x,y)$ satisfies the following estimates: There exist some $0<\varepsilon\le 1$ and $d(x,y)\not=0,$ such that
\begin{equation}\label{1.1}
|K(x,y)|\le C \frac{1}{(d(x,y))^N},
\end{equation}
\begin{equation}\label{1.2}
|K(x,y)-K(x',y)|\le C \frac{d(x,x')^\varepsilon}{(d(x,y))^{N+\varepsilon}} \quad \text{if }\ d(x,x')\leq \frac{1}{2}d(x,y),
\end{equation}
\begin{equation}\label{1.3}
|K(x,y)-K(x,y')|\le C
\frac{d(y,y')^\varepsilon}{(d(x,y))^{N+\varepsilon}} \quad \text{if }\ d(y,y')\leq \frac{1}{2}d(x,y),
\end{equation}
and furthermore, for $f, g\in C_{G,0}^\eta(\R^N)$,
$$\langle T(f), g\rangle =\int\limits_{\R^N}\int\limits_{\R^N} K(x,y) f(y) g(x)dydx$$
whenever the integral is absolutely convergent.
\end{definition}

As previously mentioned, the kernel $K(x,y)$ remains invariant under the reflection group $G$'s actions.
It is a well-established fact that classical singular kernels exhibit singular behavior along the diagonal $\Delta=\{(x,x):x\in \R^N\}$. However, the singular kernels as defined in Definition \ref{def1.1} display their singularities on $\mathcal{O}(\Delta)=\{(x,\sigma(x)): x\in \R^N, \sigma\in G\}.$

The principal results of this paper are centered around establishing the $T1$ theorem, which serves as a criterion for the boundedness of these singular integrals associated with the reflection group.
To lay the groundwork, we must extend the definition of singular integrals from above to encompass functions within ${{C}}_{G,b}^\eta(\R^N),$ the space of bounded $G$-invariant H\"{o}lder functions. The approach here is to define $T(f)$ for $f\in
{{C}}_{G,b}^\eta(\R^N)$ as a distribution over ${{C}}^\eta_{G,0,0}(\R^N)=\{ g\in {{C}}^\eta_{G,0}: \int_{\R^N} g(x)dx=0\}.$
To achieve this, given $g\in {{C}}^\eta_{G,0,0}(\R^N)$ with support contained in $\mathcal{O}_{B}(x_0, R)$ for some $x_0\in \R^N$ and $R > 0$,
we utilize a function $\xi$ such that $\xi (y)= 1$ for $y\in
\mathcal{O}_{B}(x_0, 2R)$  and $\xi(y) = 0$ for $y\in \big(\mathcal{O}_{B}(x_0, 4R )\big)^c.$
We decompose $f(y)=\xi(y)f(y) + (1-\xi(y))f(y)$ and formally have $\langle Tf,
g\rangle=\langle T(f\xi), g\rangle +\langle T(1-\xi)f, g\rangle.$
The term $\langle T(f\xi), g\rangle$ is well-defined. Due to the cancellation property of $g$ and the fact that $d(x,y) > 0$ for $x\in \supp g$ and $y\in \supp (1-\xi),$ we can express
$$\langle T(1-\xi)f, g\rangle
=\int\limits_{\R^N}\int\limits_{\R^N}
[K(x,y)-K(x_0,y)](1-\xi(y))f(y)g(x)dy dx.$$
Notice that for $x\in \mathcal{O}_{B}(x_0, R)$ and $y\notin \mathcal{O}_{B}(x_0, 2R),$ it holds that $d(x,x_0)\le R\le \frac{1}{2}d(x_0,y)$. Hence,
$$\int\limits_{\mathcal{O}_{B}(x_0, R)}\int\limits_{(\mathcal{O}_{B}(x_0, 2R))^c} |K(x,y)-K(x_0,y)||g(x)|dy dx
\le C\|g\|_1,$$
which implies that $\langle Tf, g\rangle$ is well-defined. Consequently, $T(f)$ is a distribution on ${{C}}^\eta_{G,0,0}(\R^N).$
Now, the condition $T(1)=0$ means that $\langle T1,f\rangle=0$ for any $f\in {{C}}^\eta_{G,0,0}(\R^N)$. This is synonymous with the statement that
$\int\limits_{\R^N}T^*(f)(x) dx=0$
for any function $f$ within the space ${{C}}^\eta_{G,0,0}(\R^N),$ where $T^*$ represents the adjoint of the operator $T$.
The condition $T^*(1)=0$ is defined in an analogous manner.

We now define the {\it weak boundedness property} as follows.

\begin{definition}\label{def1.2}
    The singular integral operator $T$, equipped with the kernel $K(x,y)$ as specified in Definition \ref{def1.1}, is said to have the weak boundedness property if there exist positive $\eta$ and a finite constant $C$ such that the following inequality holds:
\begin{equation}\label{1.4}
    |\langle g, Tf\rangle |\le C r^N 
\end{equation}
    for all $f, g$ in $C^\eta_{G,0}(\R^N)$, provided the supports of $f$ and $g$ are contained within $\mathcal{O}_B(y_0,r)$ and $\mathcal{O}_B(x_0,r)$, respectively. Here, $x_0$ and $y_0$ are points in $\mathbb R^N$, $r$ is a positive real number, the norm of $f$ and $g$ are constrained by $\|f\|_\infty, \|g\|_\infty \leq 1$, and their respective $G$-invariant H\"older norms are bounded by $r^{-\eta}$. We denote this property as $T \in WBP_G$.
\end{definition}

One of the main results of this paper is the following theorem:

\begin{thm}\label{th1.1}
Suppose that $T$ represents a singular integral operator associated with the reflection group $G$, which adheres to the size condition \eqref{1.1} and the smoothness condition \eqref{1.2} characterized by the regularity parameter $\varepsilon$. Then $T$ can be extended to a bounded operator on $C^\eta_{G}(\mathbb R^N)$, for $0<\eta<\varepsilon$, if two conditions are met:
{\rm(1)} $T$ satisfies $T(1)=0$, and {\rm(2)} $T$ has the weak boundedness property, that is, $T\in WBP_G$.
\end{thm}

Theorem \ref{th1.1} establishes the boundedness for operators defined on the $G$-invariant H\"{o}lder spaces. For a comparable result within the classical Euclidean context, see \cite{L}.

It is worth noting that an intriguing implication of Theorem \ref{th1.1} is the invariance of $T(f)$ under the action of the reflection group $G$ for all functions $f$ in the space $C^\eta_{G}(\mathbb R^N)$. This means that
\begin{equation}\label{invariant}
T(f)(x)=T(f)(\sigma(x))\quad \text{for all }\sigma\in G \text{ and all }x\in \Bbb R^N,
\end{equation}
which plays an essential role in the demonstration of the $T1$ theorem of this paper.

Subsequently, for a function space $X(\Bbb R^N)$, we introduce the subspace $X_G(\Bbb R^N)$, which comprises the $G$-invariant elements within $X(\Bbb R^N)$, endowed with the identical norm.
For example, ${\rm BMO}_G(\R^N)$ encompasses classical BMO functions (modulo constants) that adhere to the condition $f(x)=f(\sigma(x))$ for every $\sigma \in G$, with the norm being defined as
$$ \|f\|_{{\rm BMO_G}(\R^N)}=\|f\|_{{\rm BMO}(\R^N) }<\infty.$$

Our main result of this paper is the following $T1$ theorem pertinent to reflection-invariant operators, which augments the renowned $T1$ theorem promulgated by David and Journ\'{e} \cite{DJ}.

\begin{thm}\label{th1.2}
Suppose that $T$ is the singular integral defined in Definition \ref{def1.1}. Then $T$ is bounded on $L_G^2(\R^N)$ if and only if {\rm(1)} $T(1)\in {{\rm BMO}_G}(\R^N)$ and $T^*(1)\in {{\rm BMO}_G}(\R^N);$ {\rm(2)} $T \in WBP_G$.
\end{thm}

It is widely acknowledged that if $T$ is a classical Calder\'on--Zygmund singular integral operator and is bounded on $L^2(\R^N)$, then $T$ is also bounded on $L^p(\R^N)$ for $1<p<\infty$, is of the weak type (1,1), and maps from  $L^\infty(\R^N)$ to ${\rm BMO}(\R^N).$ These assertions persist for the singular integral operators as defined in Definition \ref{def1.1}.
\begin{thm}\label{th1.3}
Suppose that $T$ is the singular integral defined operator in Definition \ref{def1.1} and is bounded on $L_G^2(\R^N)$. Then $T$ is bounded on $L_G^p(\R^N)$ for $1<p<\infty$, is of the weak type $(1,1)$ for functions $f$ in $L_G^1(\R^N)$, and maps from $L_G^\infty(\R^N)$ to ${{\rm BMO}_G}(\R^N).$
\end{thm}

Theorems \ref{th1.2} and \ref{th1.3} present a comprehensive $L^p$ framework for singular integral operators associated with reflection groups, thereby expanding the scope of the classic Calder\'{o}n-Zygmund theory to encompass scenarios characterized by reflection invariance.

We conclude this section with an examination of the proofs for Theorems \ref{th1.2} and \ref{th1.3}.

The necessity component of Theorem \ref{th1.2} aligns with the classical case in its simplicity. To demonstrate sufficiency, we must establish both the $L^2$ boundedness of $T$ and the $G$-invariance of $Tf$ for functions $f$ within $L^2_G$. To verify the $L^2$ boundedness of $T$, we introduce smooth molecule spaces that are invariant under $G$ and construct a $G$-invariant variant of approximation to the identity. We then demonstrate that singular integral operators which are invariant under $G$ are bounded on these smooth molecule spaces. This allows us to deduce Calder\'{o}n's reproducing formula, which converges in $L^2$ and smooth molecule spaces.
We establish almost orthogonality estimates and using Cotlar-Stein lemma derive the $L^2_G$ boundedness.

Contrary to the traditional setting, to conclude the proof of the $T1$ theorem, we must also confirm the $G$-invariance of $Tf$ for $f\in L^2_G$. This does not directly stem from the $G$-invariance of $K$, and proving this fact is indeed rather non-trivial. Here, we adopt an approach based on Theorem \ref{th1.1} for $G$-invariant H\"{o}lder spaces.

Finally, to substantiate Theorem \ref{th1.3}, we first devise a $G$-invariant version of the Calder\'{o}n-Zygmund decomposition, which is of independent interest. We then harness this to corroborate the weak-type $(1,1)$ nature of $T$.
Regarding the endpoint $p=\infty$, we initially extend the domain of $T$ to include $L^\infty_G$ and subsequently furnish a rigorous proof for the $L^\infty_G\to \text{BMO}_G$ estimations. The $L_G^p$ boundedness of $T$ for $1<p<\infty$ is thereafter established through interpolation techniques.

The structure of this article is as follows. In Section ~\ref{sec:2}, we establish the boundedness of reflection-invariant singular integrals on spaces of smooth molecules and derive a Calder\'{o}n's reproducing formula. The demonstration of Theorem ~\ref{th1.1} is detailed in Section ~\ref{sec:3}. Section ~\ref{sec:4} is dedicated to proving Theorems ~\ref{th1.2} and ~\ref{th1.3}. Lastly, proofs for several essential technical lemmas are provided in the appendix.

\section{{Boundedness of reflection invariant singular integral on smooth molecule spaces and Calder\'{o}n's reproducing formula}}\label{sec:2}

It is widely acknowledged that in the traditional context, almost orthogonal estimates serve as crucial instruments for substantiating the $T1$ theorem.
The forthcoming proposition furnishes an analogous instrument within our specific framework.

\begin{thm}\label{th2.1}
    Suppose that $T$ is a $G$-invariant singular integral defined in Definition \ref{def1.1}. Suppose further that $T$ is bounded on $L_G^2(\R^N), T(1)=T^*(1)=0,$ and $K(x,y)$, the kernel of $T$, satisfies the following
    \begin{equation}\label{e:double}
    |[K(x,y-K(x',y))]-[K(x,y')-K(x',y')]|\le C\left(\frac{d(x,x')d(y,y')}{d(x,y)^2}\right)^{\varepsilon} \frac{1}{d(x,y)^N}
    \end{equation}
    for $d(x,x')\le \frac{1}{2}d(x,y)$ and $d(y,y')\le \frac{1}{2}d(x,y).$
    Then $T$ maps $\mathbb M(\beta, \gamma, r, x_0)$ boundedly to itself
    with $0<\beta, \gamma<\varepsilon\le 1,$ where $\varepsilon$ is the exponent of the regularity of the kernel of $T.$ Moreover, there exists a constant $C$ such that
    $$\|Tf\|_{\M(\beta, \gamma, r, x_0)}\le C \|f\|_{\M(\beta, \gamma, r, x_0)}.$$
\end{thm}
Here $\mathbb M(\beta, \gamma, r, x_0)$ is the smooth molecule space defined by following

\begin{definition}\label{sm}
    A function $f(x)$ is said to be a smooth molecule for $0<\beta\le 1, \gamma>0, r>0$ and some fixed $x_0\in \R^N,$ if $f(x)$ satisfies the following conditions:
    \begin{equation}\label{sm1.17}
    |f(x)|\le C \frac{r^\gamma}{(r+d(x,x_0))^{N+\gamma}};
    \end{equation}
    \begin{equation}\label{sm1.18}
    \begin{aligned}
    |f(x)-f(x')|&\le C  \Big(\frac{d(x,x')}{r+d(x,x_0)}\Big)^\beta
    \frac{r^\gamma}{(r+d(x,x_0))^{N+\gamma}}
    \end{aligned}
    \end{equation}
    for $d(x,x')\le \frac{1}{2}(r+d(x,x_0))$ and
    \begin{equation}\label{sm1.19}
    \int\limits_{\R^N} f(x) dx=0.
    \end{equation}
    If $f(x)$ is a smooth molecule, we write $f\in \mathbb M(\beta, \gamma, r, x_0)$ and define the norm of $f$ by
    $\|f\|_{\mathbb M(\beta, \gamma, r, x_0)}:=\inf\{C: (\ref{sm1.17})-(\ref{sm1.18})\ {\rm hold}\}.$
\end{definition}
It can be easily verified that $\mathbb M(\beta, \gamma, r, x_0)$  is a Banach space. We remark that all smooth molecules $f$ are $G$-invariant, that is, $f(x)=f(\sigma (x))$ for all $\sigma\in G.$ 
\begin{proof}[{\bf Proof of Theorem \ref{th2.1}}]
Let us consider a function $f\in \mathbb M(\beta, \gamma, r, x_0)$ with  $\|f\|_{{\mathbb M}(\beta, \gamma, r,
        x_0)}=1$, where $\beta$ and $\gamma$ are positive parameters with values less than $\varepsilon$, and $\varepsilon$ denotes the order of regularity for the kernel associated with the operator $T$. Our aim is to demonstrate that the norm $\|T(f)\|_{\mathbb M(\beta, \gamma, r, x_0)}$ is bounded above by a constant $C$.

Initially, we assess the size requirement for $Tf(x)$. To achieve this, we consider two cases: Case (1): $d(x,x_0)\le 5r$ and
    Case (2): $d(x,x_0)=R> 5r.$

For Case (1), we define $1=\xi(y)+\eta(y)$, where $\xi(y)$ is represented as $\theta\Big(\frac{d(y,x_0)}{10r}\Big)$ with $\theta$ belonging to $C_0^\infty(\mathbb{R})$. The function $\theta$ is characterized by $\theta(x)=1$ for $|x|\le 1$ and $\theta(x)=0$ for $|x|\ge 2$. Since $T$ is bounded on $L_G^2(\R^N),$ we can write
\begin{align*}
    Tf(x)
    &= \langle K(x,\cdot),(\xi+\eta)f\rangle
    \\&=  \int\limits_{\R^N} K(x,y)\xi(y)(f(y)-f(x))dy+f(x)\langle
    K(x,\cdot),\xi(\cdot)\rangle
    \\&\qquad +\int\limits_{\R^N} K(x,y)\eta(y)f(y)dy
    =:  I_1+I_2+I_3.
    \end{align*}

    Utilizing the size condition for the kernel $K(x,y)$ as specified in Definition \ref{def1.1}, in conjunction with the smoothness requirement \eqref{sm1.18} for $f$, we obtain
    \begin{align*}|I_1|
    &\lesssim  \int\limits_{d(x,y)\le 20r} |K(x,y)|\cdot|f(y)-f(x)|dy
    \\&\lesssim  \int\limits_{d(x,y)\le 20r} \frac{1}{d(x,y)^N}\Big(\frac{d(x,y)}{r}\Big)^\beta
    \Big\{\frac{r^\gamma}{(r+d(y,x_0))^{N+\gamma}}
    + \frac{r^\gamma}{(r+d(x,x_0))^{N+\gamma}}\Big\} dy.
    \end{align*}

Note that the conditions $d(y,x)\le 20r$ and $d(x,x_0)\le 5r$ implies $r+d(y,x_0)\sim r+d(x,x_0)$. Consequently,
    \begin{align*}|I_1|
    &\lesssim  \frac{1}{r^\beta}\frac{r^\gamma}{(r+d(x,x_0)^N}\int\limits_{d(x,y)\le 20r}\frac{1}{d(x,y))^N}d(x,y)^\beta dy
    \lesssim \frac{r^\gamma}{(r+d(x,x_0))^{N+\gamma}}.
    \end{align*}

To evaluate $I_2$, it is sufficient to demonstrate that $|T\xi(x)|\lesssim 1$ for $d(x,x_0)\le 5r$.
To this end, consider a function $\psi$ belonging to $C^\eta_G(\mathbb R^N)$ with its support restricted to $\mathcal{O}_B(x_0,5r)$ and satisfying the condition that $\int\limits_{\R^N}\psi(x)dx=0$. By the facts that $T(1)=0$ and $\int\limits_{\R^N}\psi(x)dx=0$, we have
    \begin{align*}
    |\langle T\xi, \psi\rangle| &=|- \langle T\eta,\psi\rangle|
    =\Big|\int\limits_{\R^N}\int\limits_{\R^N} [K(x,y)-K(x_0,y)]\eta(y)\psi(x)dydx\Big|.
    \end{align*}
    Observe that the supports of $\eta(y)$ and $\psi(x)$ imply $d(y,x_0)> 10r$ and $d(x, x_0)\le 5r,$ respectively. The smoothness condition of $K$ yields
    \begin{align*}
    |\langle T\xi, \psi\rangle| \lesssim \iint\limits_{d(y,x_0)> 2d(x, x_0)}\frac{1}{d(y,x_0)^N} \Big(\frac{d(x, x_0)}{d(y, x_0)}\Big)^\varepsilon dy|\psi(x)|dx
    \lesssim  \int\limits_{\R^N} |\psi(x)|dx.
    \end{align*}
    This implies that $T\xi(x)=\alpha+\gamma(x)$ for $x\in \mathcal{O}_B(x_0,5r)$, where $\alpha$ is a constant depending on $\xi$ and $\|\gamma\|_\infty\le C_0$ for some constant $C_0$ independent of $\xi.$
    To estimate $\alpha$, choose $\varphi
    \in C_{G,0}^\eta(\mathbb{R}^N)$ with supp $\varphi\subseteq \mathcal{O}_B(x_0,5r),  0\le \varphi\le 1, \|\varphi\|_\eta\le r^{-\eta}$
    and $\int_{\R^N} \varphi(x)dx=C_1 r^N,$ for some constant $C_1$ independent of $r.$ We then use our assumption that  $T$ is bounded on $L_G^2(\R^N)$ to get
    $$\bigg|C_1 r^N\alpha + \int\limits_{\R^N} \varphi(x)\gamma(x)dx\bigg|=|\langle T\xi,\varphi\rangle |\le C r^N,$$
    which implies $|\alpha|\le C_0+\frac{C}{C_1}$ and hence,
    \begin{equation}\label{new1}
    |T\xi(x)|\lesssim 1
    \end{equation}
    for $x\in \mathcal{O}_B(x_0,5r).$ Thus,
    $$|I_2|\lesssim |f(x)|\lesssim \frac{r^\gamma}{(r+d(x,x_0))^{N+\gamma}}.$$

    For the last term $I_3$, observe that $d(x,x_0)\le 5r$ and the support of $\eta(y)$ is contained in $\{y\mid d(y,x_0)\ge 10r\},$ so $d(x,y)\ge 5r$ and $d(x,y)\sim d(y,x_0),$ and thus,
    \begin{align*}|I_3|
    &\lesssim  \int\limits_{d(y,x_0)\ge 10 r\atop d(x,y)\ge 5r} \frac{1}{d(x,y)^N}\frac{r^\gamma}{(r+d(y,x_0))^{N+\gamma}} dy
    \\&\lesssim  \frac{1}{r^N}        \lesssim \frac{r^\gamma}{(r+d(x,x_0))^{N+\gamma}}.
    \end{align*}
This concludes all estimates for Case (1).

  We now treat Case (2) where $d(x,x_0)=R > 5r$.
Set $1=I(y)+J(y)+L(y)$, where $I(y)=\theta\big(\frac{16d(y,x)}{R}\big), J(y)=\theta\big(\frac{16d(y,x_0)}{R}\big)$ and $f_1(y)=I(y)f(y),f_2(y)=J(y)f(y)$ and
    $f_3(y)=L(y)f(y).$

    Observe that if $y$ is in the support of $f_1,$ then $d(y,x_0)\sim d(x,x_0)=R,$ and thus,
    \begin{align*}
    \textup{(i)}&\ |f_1(y)|  \lesssim \frac{r^\gamma}{R^{N+\gamma}};
    \\
    \textup{(ii)}&\ \int\limits_{\R^N} |f_1(y)|dy  \lesssim  \int\limits_{d(y,x_0)\ge \frac{7R}{8}}\frac{r^\gamma}{(r+d(y,x_0))^{N+\gamma}} dy
    \lesssim \Big(\frac{r}{R}\Big)^\gamma;\\
    \textup{(iii)}&\ |f_1(y)-f_1(x)|\lesssim  \Big(\frac{d(y,x)}{R}\Big)^\beta \frac{r^\gamma}{(r+d(y,x_0))^{N+\gamma}};
    \\
    \textup{(iv)}&\ \int\limits_{\R^N} |f_3(y)| dy  \lesssim   \int\limits_{d(y,x_0)\ge \frac{R}{16}} \frac{r^\gamma}{(r+d(y,x_0))^{N+\gamma}}dy
    \lesssim \Big(\frac{r}{R}\Big)^\gamma.
    \end{align*}

Given the fact that $\int\limits_{\R^N} f(y)dy=0$, it follows that
    \begin{align*}
    \textup{(v)}\ \Big|\int\limits_{\R^N} f_2(y)dy\Big|  =\Big|-\int\limits_{\R^N} f_1(y)dy -\int_{\R^N} f_3(y)dy \Big|
    \lesssim \Big(\frac{r}{R}\Big)^\gamma.\ \ \ \ \ \ \ \ \ \ \ \ \ \ \ \ \ \
    \end{align*}

We begin by estimating $Tf_1(x)$ as follows.
Let $u(y)=\theta\Big(\frac{2d(y,x)}{R}\Big)$. Then $f_1(y)=u(y)f_1(y).$
Our subsequent step is to consider
    \begin{align*}Tf_1(x)
    &=  \langle K(x,\cdot), u (\cdot) f_1(\cdot)\rangle
    =  \int\limits_{\R^N} K(x,y)u(y)[f_1(y)-f_1(x)]dy+f_1(x)\langle
    K(x,\cdot),u(\cdot)\rangle
    \\&=:  I\!I_{1.1}+I\!I_{1.2}.
    \end{align*}

Utilizing a similar method as employed for  $I_2,$ we can derive the inequality $|T(u)(x)|\lesssim 1.$ This implies that
$$ |I\!I_{1.2}|\lesssim |f(x)|\lesssim \frac{r^\gamma}{(r+d(x,x_0))^{N+\gamma}}.$$

For the term $I\!I_{1.1}$, applying the size condition on the kernel $K(x,y)$ and  property \textup{(iii)} above  leads to the following estimate:
\begin{align*}
|I\!I_{1.1}|
\lesssim  \int\limits_{d(x,y)\le R} \frac{1}{d(x,y)^N} \Big(\frac{d(x,y)}{R}\Big)^\beta\Big(\frac{r^\gamma}{R^{N+\gamma}}\Big)^\gamma  dy
    \lesssim \frac{r^\gamma}{(r+d(x,x_0))^{N+\gamma}}.
    \end{align*}
The derived inequality demonstrates the desired bound for the term $I\!I_{1.1}$.

To treat $Tf_2(x)$, we partition it into two components:
\begin{align*}Tf_2(x)
        =  \int\limits_{\R^N} [K(x,y)-K(x,x_0)]f_2(y)dydy+K(x,x_0)\int\limits_{\R^N} f_2(y)dy
        =:  I\!I_{2.1}+I\!I_{2.2}.
    \end{align*}

    Using the estimate provided in (v), we obtain
    \begin{align*}
    |I\!I_{2.2}|
        \lesssim  |K(x,x_0)|\Big(\frac{r}{R}\Big)^{\gamma}
        \lesssim   \frac{1}{d(x,x_0)^N}\Big(\frac{r}{R}\Big)^{\gamma}
        \lesssim \frac{r^\gamma}{(r+d(x,x_0))^{N+\gamma}}.
    \end{align*}

To estimate the term $I\!I_{2.1}$, we invoke the smoothness condition of $K(x,y)$ as stipulated in Definition \ref{def1.1} with $d(y,x_0)\le \frac{1}{2}d(x,x_0)$, we derive the following:
\begin{align*}|I\!I_{2.1}|
    &\lesssim  \int\limits_{d(y,x_0)\le \frac{R}{8}} \frac{1}{d(x,x_0)^N}\Big(\frac{d(y,x_0)}{d(x,x_0)}\Big)^{\varepsilon}  \frac{r^\gamma}{(r+d(y,x_0))^{N+\gamma}}dy
    \\&\lesssim  \frac{r^\gamma}{R^{N+\gamma}}\lesssim  \frac{r^\gamma}{(r+d(x,x_0))^{N+\gamma}}.
    \end{align*}

    Finally,
    \begin{align*}|Tf_3(y)|
    &\lesssim  \int\limits_{ d(y,x)\ge \frac{R}{16}, \atop d(y,x_0)\ge \frac{R}{16}}\frac{1}{d(x,y)^N}\frac{r^\gamma}{(r+d(y,x_0))^{N+\gamma}}dy
    \\&\lesssim   \frac{1}{d(x,x_0))^N} \int\limits_{  d(y,x_0)\ge \frac{R}{16}}\frac{r^\gamma}{(r+d(y,x_0))^{N+\gamma}}dy
    \lesssim  \frac{r^\gamma}{(r+d(x,x_0))^{N+\gamma}}.
    \end{align*}

    It remains to show the regularity of $T(f),$ that is the following estimate:
    \begin{align*}
    |Tf(x)-Tf(x')|\lesssim
    \Big(\frac{d(x,x')}{r+d(x,x_0)}\Big)^\beta \frac{r^\gamma}{(r+d(x,x_0))^{N+\gamma}}
    \end{align*}
    for $d(x,x')\le \frac{1}{2}{(r+d(x,x_0)).}$

To proceed with the regularity proof for $T(f)$, we set $d (x,x_0)=
    R$ and $d (x, x^\prime )=\delta.$ We first consider the case where $R \geq 10r$ and $\delta \leq \frac{1}{20}(r + R)$. As before, let us decompose $1 = I(y) + J(y) + L(y),$ with $I(y) =\theta
    ({{16d (y,x)}\over{R}}), J(y) =\theta ({{16d
            (y,x_0)}\over{R}}),$ and define $f_1(y)= f(y)I(y), f_2(y)= f(y)J(y),$ and
    $f_3(y)= f(y)L(y).$ We can express
\begin{align*}
        T(f_1)(x)&=\int\limits_{\R^N} K(x,y)u(y)[f_1(y) - f_1(x)] dy\\
        & \qquad + \int\limits_{\R^N} K(x,y)v(y)f_1(y) dy+f_1(x)\int\limits_{\R^N} K(x,y)u(y)dy,
    \end{align*}
where  $u(y) =\theta ({{d (x,y)}\over{2\delta }})$ and $v(y) = 1
    - u(y).$ Let us denote the first term on the right-hand side by $p(x)$ and the sum of the last two terms by $q(x)$. The size condition of $K$ and the smoothness of $f_1$ imply
\begin{align*}
        \vert  p(x)\vert &\lesssim \int\limits_{d (x,y)\leq 4\delta
        }\frac{1}{d (x,y)^{N}}{{d (x,y)^\beta }\over{R^\beta }}{{r^\gamma
            }\over{R^{N+\gamma }}} dy \lesssim {{\delta ^\beta }\over{R^\beta }}{{r^\gamma }\over{R^{N+ \gamma
        }}}.
    \end{align*}
This estimate remains valid when $x$ is replaced by $x'$ as $d(x, x') = \delta$. Consequently, we obtain
$$\vert  p(x) - p(x^\prime )\vert \le |p(x)|+|p(x')|\lesssim {{\delta ^\beta }\over{R^
            \beta }}{{r^\gamma }\over{R^{1+\gamma }}}.$$

    To derive the regularity for $q(x)$, we use the condition $T(1) = 0$ to write:
\begin{align*}
        q(x) - q(x^\prime )&=\int\limits_{\R^N} [K(x,y) - K(x^\prime ,y)]v(y)[f_1(y) -
        f_1(x)] dy \\
        & \qquad +[f_1(y) - f_1(x)]\int\limits_{\R^N} K(x,y)u(y) dy =: I\!I\!I_{1} + I\!I\!I_{2}.
    \end{align*}
Utilizing an argument similar to that employed in the derivation of \eqref{new1}, along with the estimate (iii) for $f_1$, we deduce
    $$\vert  I\!I\!I_{2} \vert \lesssim \vert f_1(x) - f_1(x^\prime )\vert
    \lesssim {{\delta
            ^\beta }\over{R^\beta }}{{r^\gamma }\over{R^{N+ \gamma }}}.$$

Furthermore, upon noting
$$\vert f_1(y) - f_1(x)\vert \vert  v(y)\vert \lesssim
    {{ d
            (x,y)^\beta }\over{R^\beta }}{{r^\gamma }\over{R^{N+\gamma }}}$$
for all $y\in \R^N,$  it becomes apparent that $I\!I\!I_{1}$ is majorized by
    \begin{align*}
       |I\!I\!I_{1}| &\lesssim \int\limits
        _{d (x,y)\geq 2\delta }{{d (x,x^ \prime )^\varepsilon
            }\over{d (x,y)^{N+\varepsilon}}}{{d (x,y)^\beta
            }\over{R^\beta }} \frac{r^\gamma }{R^{N+\gamma }} dy  \lesssim {{\delta ^\beta }\over{R^\beta }}{{r^\gamma }\over{R^{N+ \gamma}}}
    \end{align*}
since $\beta < \varepsilon$. We have proved that
$$\vert  T(f_1)(x) - T(f_1)(x^\prime )\vert \lesssim {{\delta
            ^\beta }\over{R^\beta }}{{r^\gamma }\over{R^{N+\gamma }}}.$$

It is important to note that for $d (x,x^\prime )=\delta \leq {{1}\over{20}}( r +
    R)$ and $R \geq 10r$, the points $x$ and $x'$ do not lie within the support of $f_2$ and $f_3$.
By invoking the properties of $K$ and the estimate for $f_2$ as specified in (v), we can derive the following inequality:
\begin{align*}
        & \vert  T(f_2)(x) - T(f_2)(x^\prime )\vert =\bigg\vert
        \int [ K(x,y) - K(x^\prime,y)]f_2(y) dy\bigg\vert\\
        &\qquad \leq \int \vert  K(x,y) - K(x^\prime,y) - K(x,x_0) - K(x^ \prime
        ,x_0)\vert \vert f_2(y)\vert  dy\\
        &\qquad\quad +\vert K(x,x_0) - K(x^\prime ,x_0)\vert \bigg\vert \int f_2(y) dy
        \bigg\vert \\
        &\qquad \lesssim
       \int\limits_{d ({x_0},y)\leq {{R}\over{4}}} {{d
                (x,x^\prime )^\varepsilon d (y,x_0)^\varepsilon }\over{R^{2+\varepsilon}}}{{r^\gamma }\over{d(y,x_0)^{N+\gamma }}} dy+{{\delta ^\varepsilon }\over{R^{N+\varepsilon }}} \frac{r^\gamma
        }{R^\gamma } \\
        & \qquad \lesssim {{\delta ^\varepsilon }\over{R^\varepsilon }}{{r^\gamma }
            \over{R^{N+\gamma }}}.
    \end{align*}

    Finally,
    \begin{align*}
     \vert  T(f_3)(x) - T(f_3)(x^\prime )\vert &=\bigg\vert
        \int [ K(x,y) - K(x^\prime ,y)]f_3(y) dy\bigg\vert\\
        & \lesssim \int\limits_{d (x,y)\geq {{R}\over{8}}\geq 2\delta
        }{{d (x,x^\prime )^\varepsilon }\over{d (x,y)^{N+\varepsilon}
        }}\vert f_3(y)\vert  dy
        \lesssim {{\delta ^\varepsilon }\over{R^ \varepsilon
        }}{{r^\gamma }\over{R^{N+\gamma }}}.
    \end{align*}
These inequalities demonstrate that the operator $T(f)(x)$ satisfies the required smoothness properties when $d (x,x_0)=
    R\geq 10 r$ and $d (x,x^\prime )=\delta \leq {{1}\over{20}}( r
    + R).$

We now consider the other cases. Note first that when $d(x, x_0) = R$ and ${{1}\over{2}}( r + R)\geq d (x,x^\prime )=\delta
    \geq {{1}\over{20}}( r + R),$ the smoothness estimate for $T(f)(x)$ follows directly from the size estimate of $T(f)(x)$:
\begin{align*}
    |Tf(x)-Tf(x')|
    \lesssim |Tf(x)|+|Tf(x')|\lesssim  \frac{r^\gamma}{(d(x,x_0))^{N+\gamma}}.
    \end{align*}
Therefore, it suffices to examine the situation where $R\leq 10 r$ and $\delta
    \leq {{1}\over{20}}( r + R).$ This case is analogous and indeed more straightforward. The primary adjustment required is to simply substitute $R$ with $r$ in the proof sequence detailed above. We will omit the repetition of these details here. The proof of Theorem \ref{th2.1} is concluded.
\end{proof}

In applying Theorem \ref{th2.1}, we aim to establish Calder\'{o}n's reproducing formula pertinent to reflection groups. The endeavor begins with the construction of a $G$-invariant approximation of the identity, drawing inspiration from Coifman's methodology, as referenced in \cite{DJS}. To proceed with precision, let us consider a function $h$ belonging to $C^1(\mathbb{R})$ such that $h(t)$ equals 1 for $|t|\leq 1$, vanishes when $|t|\geq 2$, and  $h(t)\in [0, 1]$ for all $t\in\mathbb{R}$. For every integer $k\in\mathbb{Z}$, we define
$$
 T_k ( f )(x)=\int\limits_{\R^N} h(2^{k}d(x,y)) f(y)dy.
$$
It is evident that
\begin{equation}\label{n2.6}
|\mathcal{O}_B(x,2^{-k})| \leq T_k(1)(x)\leq |\mathcal{O}_B(x,2^{1-k})|,
\end{equation}
 which implies that $T_k(1)(x)$ is comparable to $2^{-kN}$ as $|\mathcal{O}_B(x,r)|\sim|B(x,r)|\sim r^N.$ Consequently,
$$
T_k\left( \frac{1}{T_k(1)} \right)(x)=\int\limits_{\R^N} h(2^{k}d(x,y)) \frac{1}{T_k(1)(y)}\,dy \approx 1.
$$
We proceed to define two multiplication operators $M_k(f)(x)=\big( T_k(1)(x)\big)^{-1}f(x)$ and $W_k(f)(x)=\big(T_k( {1\over T_k(1)})(x)\big)^{-1}f(x)$, and the operator $S_k(f)(x)=M_kT_kW_kT_kM_k(f)(x).$ We claim that $S_k(x,y)$, the kernel of $S_k$, satisfies the following conditions:
\begin{eqnarray*}
    &\textup{\bf (i)}& S_k(x,y) = 0 {\rm\ for\ } d(x,y) \geq 2^{2-k},
    {\rm\ and\ } \| S_k\|_{\infty} \lesssim 2^{kN},\\
    &\textup{\bf (ii)}& |S_k(x,y)-S_k(x',y)|
    \lesssim {d(x,x')\over 2^{-k}} 2^{kN},\\
    &\textup{\bf (iii)}& |S_k(x,y)-S_k(x,y')|
    \lesssim {d(y,y')\over 2^{-k}} 2^{kN},\\
    &\textup{\bf (iv)}& |[S_k(x,y)-S_k(x',y)]-[S_k(x,y')-S_k(x',y')]|
    \lesssim {d(x,x')\over 2^{-k}}{d(y,y')\over 2^{-k}} 2^{kN},\\
    &\textup{\bf (v)}& \int\limits_{\R^N}S_k(x,y) dy
    = \int\limits_{\R^N}S_k(x,y) dx=1,\\
    &\textup{\bf (vi)}& \text{for all}\ f\in L^2_G(\R^N), \lim\limits_{k\to +\infty}\Big\|\int\limits_{\R^N}S_k(\cdot,y)f(y) dy-f(\cdot)\Big\|_2=0,\\
    &\textup{\bf (vii)}&\text{for all}\ f\in L^2_G(\R^N), \lim\limits_{k\to -\infty}\Big\|\int\limits_{\R^N}S_k(\cdot,y)f(y) dy\Big\|_2=0.
\end{eqnarray*}

To verify the aforementioned claim, observe that $S_k(x,y)=S_k(y,x)$  and $S_k(1)=1$. Thus, condition {\bf (v)} is satisfied. To verify condition {\bf (i)}, we express
\begin{align*}
    S_k(x,y)={1\over T_k(1)(x)}\bigg\{ \int\limits_{\R^N} h(2^{k}d(x,z)) {1\over T_k\Big( {1\over T_k(1)} \Big)(z) } h(2^{k}d(z,y))dz \bigg\}{1\over T_k(1)(y)}.
\end{align*}
Given the support conditions for $h$, if $S_k(x,y)\neq 0$, then $d(x,y)\leq 2^{2-k}$. The desired estimate for $L^\infty$ norm of $S_k$ follows from \eqref{n2.6}, establishing {\bf (i)}.

Considering condition {\bf (ii)}, by virtue of {\bf (i)}, we need only examine the case where $d(x,x')\leq 2^{-k}.$  Indeed, if $d(x,x')>2^{-k}$ then {\bf (ii)} directly follows from {\bf (i)}.

To establish {\bf (ii)} for the case where $d(x,x')\leq 2^{-k}$, we write
\begin{align*}
    &S_k(x,y)-S_k(x',y)\\
    &=\Big[{1\over T_k(1)(x)}-{1\over T_k(1)(x')}\Big]\bigg\{ \int\limits_{\R^N} h(2^{k}d(x,z)) {1\over  T_k\Big( {1\over T_k(1)} \Big)(z) } h(2^{k}d(z,y)) dz \bigg\}{1\over T_k(1)(y)}\\
    &\hskip.2cm+{1\over T_k(1)(x')}
    \bigg\{ \int\limits_{\R^N} [h(2^{k}d(x,z))-h(2^{k}d(x',z))] {1\over  T_k\Big( {1\over T_k(1)} \Big)(z) } h(2^{k}d(z,y)) dz \bigg\}{1\over T_k(1)(y)}\\
    &=: Z_1+Z_2.
\end{align*}

For $d(x,x')\leq 2^{-k},$ by the regularity on the function $h,$ we have
\begin{align*}
    \Big|{1\over T_k(1)(x)}-{1\over T_k(1)(x')}\Big|= \Big|\frac{\int_{\R^N} (h( 2^{k}d(x,y))-h(2^{k}d(x',y)))dy}{T_k(1)(x)T_k(1)(x')}\Big|
    \lesssim 2^{kN}\cdot 2^{k} d(x,x'),
\end{align*}
which implies that
$$
|Z_1|\lesssim \frac{d(x,x')}{2^{-k}}2^{kN} \, .
$$
The estimate for $Z_2$ is similar, the details being omitted.

The proof of {\bf (iii)} is similar to that for {\bf (ii)}, and {\bf (iv)} follows from {\bf (ii)} and {\bf (iii)}.

Now we check {\bf (vi)}. Observe if $f$ is a continuous function with compact support and $f(x)=f(\sigma(x))$ then
$\int\limits_{\R^N}S_k(x,y)f(y)dy$ converges to $f(x)$ for almost every $x\in\R^N$ as $k\to +\infty.$ This fact together with the following maximal function estimate
$$\sup\limits_{k\in\mathbb{Z}}\Big|\int\limits_{\R^N}S_k(x,y)f(y)dy\Big|\lesssim M(f)(x),$$
where $M(f)(x)$ is the  uncentered Hardy-Littlewood maximal function, implies {\bf (vi)} since the subspace of $G-$invariant  continuous functions with compact support is dense in $L_G^2(\R^N)$ (Lemma \ref{le4.3} in the appendix).

Finally the proof of {\bf (vii)} follows from the following estimate:
$$\Big|\int\limits_{\R^N}S_k(x,y)f(y) dy\Big|\lesssim {2^{kN/2}} \|f\|_{L_G^2} \to 0, \text{ as } k \to -\infty.$$
The claim is thus proved.
We remark that $S_k(x,y)$ is invariant under the actions of the reflection group $G$.

Let $D_k=S_k - S_{k-1}$ and $D_k(x,y)$ be the kernel for $k\in \Z$. Then
$D_k(\cdot ,y)\in \M(\varepsilon,\varepsilon,2^{-k},y),$ where $0 < \varepsilon \leq 1$ for any fixed $y$ and $k$,
and similarly, $D_k(x,\cdot)\in \M(\varepsilon,\varepsilon,2^{-k},x)$ for fixed $x$ and $k$.

We need the following estimates for the approximation to the identity; see \cite{DH} for further details.
\begin{lem} \label{d-difference} 
    Let $0<\varepsilon \leq 1$. Suppose that $\{S_k\}_{k\in \mathbb{Z}}$ is an approximation to the identity. Set $D_k=S_k - S_{k-1}$ for all $k\in \mathbb Z.$ Then for  $0<\varepsilon^\prime <\varepsilon,$  there exists a constant $C$ which depends on $\varepsilon^\prime$ and $\varepsilon$, but not on $k, l$ such that
    \begin{equation}\label{e3.1.18}
    | D_kD_l(x, y)| \leq C 2^{-|k-l|\varepsilon^\prime}{2^{{-(k\wedge l )}{\varepsilon}
        }\over{(2^{{-(k\wedge l )}} + d(x,y))^{N+\varepsilon}}};
    \end{equation}
    for $d(x,x')\leq \frac{1}{2}(2^{-(k\wedge l)}
    +d(x,y))$
    \begin{equation}\label{e3.1.19}
    \quad | D_kD_l(x, y) - D_kD_l(x^\prime, y)|\leq
    C\Big(\frac{d(x,x')}{2^{-(k\wedge l)} + d(x,y)}\Big)^{\varepsilon } \frac{2^{-(k\wedge l)\varepsilon}} {(2^{-(k\wedge l)} + d(x,y))^{N+\varepsilon}}.
    \end{equation}
    Similar estimates hold whenever $x$ and $y$ are interchanged.

    \begin{align}\label{(e3.1.23)}
\begin{split}
    &| D_kD_l(x, y) - D_kD_l(x^\prime, y) - D_kD_l(x,y^\prime
    )+D_kD_l(x^\prime, y^\prime )| \hspace{-1cm} \\
    & \quad \leq C\Big(\frac{d(x,x')}{2^{-(k\wedge l)} + d(x,y)}\Big)^{\varepsilon^\prime }
    \Big(\frac{d(y, y^\prime)}{2^{-(k\wedge l)}
        + d(x,y)}\Big)^{\varepsilon^\prime }
    \frac{2^{-(k\wedge l)(\varepsilon-\varepsilon^\prime) }} {(2^{-(k\wedge l )} + d(x,y))^{N+\varepsilon-\varepsilon^\prime }}
\end{split}
    \end{align}
    for $d(x,x')\leq \frac{1}{2}(2^{-(k\wedge l )} +d(x,y))$ and $d(y,y')\leq \frac{1}{2} ( 2^{-(k\wedge
        l)}+d(x,y)),$ where $k\wedge l=\min(k,l).$
\end{lem}

Let $x_0\in \Bbb R^N$ be fixed and define ${\M}(\beta,\gamma )=
{\M}(x_0, 1, \beta,\gamma ).$ It is evident that for any $x_1 \in \mathbb{R}^N$ and $r > 0$, we have  ${\M}(x_1, r, \beta,\gamma )={\M}(\beta, \gamma )$ with equivalent norms.

The subsequent key result of this section is the Calder\'{o}n's reproducing formula.

\begin{thm}\label{th2.2}
Suppose that $\lbrace D_k(x,y)\rbrace _{k\in
\mathbb Z}$ are same as given in Theorem \ref{th2.1}. Then there exist two families of operators, namely $\lbrace {\widetilde D}_k\rbrace _{k\in
\mathbb Z}$ and$\lbrace{\widetilde {\widetilde D}}_k\rbrace _{k\in
\mathbb Z}$, with the property that for any $f\in {\M}(\beta,\gamma ),$ the following equality holds:
\begin{align}\label{(e3.2.2)}
f =\sum_{k\in \Bbb Z} {\widetilde D}_kD_k(f)= \sum_{k\in \Z} D_k{\widetilde{\widetilde D}}_k(f),
\end{align}
where the series converge in the norm of ${\M}(\beta ^\prime
,\gamma ^\prime )$, with $0<\beta ^\prime <\beta $ and $0<\gamma
^\prime <\gamma.$ Furthermore, the kernels ${\widetilde D}_k(x,y)$ of the operators ${\widetilde
D}_k,$ considered as functions of $x$, belong to the space ${\M}(y,2^{ -k},\varepsilon ^\prime,\varepsilon ^\prime )$ for any fixed $y$. These kernels satisfy the cancellation conditions
$$\int\limits_{\R^N} {\widetilde D}_k(x,y) dy=\int\limits_{\R^N} {\widetilde D}_k(x, y) dx= 0,\qquad \forall\, k\in \mathbb Z.$$
for all  $k\in \mathbb Z$. Similarly, the kernels ${\widetilde{\widetilde D}}_k(x,y)$ of the operators ${\widetilde{\widetilde D}}_k,$ viewed as functions of $y$, are elements of  ${\M}(x,2^{ -k},\varepsilon ^\prime,\varepsilon ^\prime )$ for any fixed $x$ and also fulfill the cancellation conditions
$$\int\limits_{\R^N} {\widetilde{\widetilde D}}_k(x, y) dy=\int\limits_{\R^N} {\widetilde{\widetilde D}}_k(x, y) dx= 0,\qquad \forall \, k\in \mathbb Z.$$

\end{thm}

\begin{proof}
Let $\{S_j\}$ be an approximation to the identity that satisfies the double Lipschitz condition, as well as the properties
    \begin{align*}
        \lim_{j\to \infty} S_j&=I,\qquad \text{the identity operator on }L_G^2(\R^N)\\
        \lim_{j\to -\infty} S_j&=0,
    \end{align*}
both in the strong operator topology on $\mathscr{B}(L^2_G(\R^N)).$ Set $D_j=S_{j+1}-S_j$. We apply Coifman's decomposition of the identity operator as follows: For any large positive integer $M$,
\begin{align*}
    f(x)&=\lim_{j\to \infty} S_j(f)(x)=\sum_{j\in \Z} D_j(f)(x)=\sum_{j\in \Z} \sum_{k\in \Z} D_k D_j (f)(x)\\
    &=\sum_{|j-k|\le M} D_k D_j (f)(x)+\sum_{|j-k|>M} D_k D_j(f)(x).
    \end{align*}
We note that David-Journ\'{e} and Semmes utilized the aforementioned Coifman decomposition in their work \cite{DJS} to demonstrate the $T1$ and $Tb$ theorems on spaces of homogeneous type.

To prove Theorem \ref{th2.2}, we need the following preliminaries.

\begin{prop} \label{proposition 3.2.3}
Let $\lbrace  S_k\rbrace _{k\in \mathbb Z}$ be an approximation to the identity and $M$ be a sufficiently large positive integer. Define $D_k=S_k - S_{k-1}$ for all $k\in \mathbb Z.$  Consider the operator $T_M= \sum_kD^M_kD_k$, where $D^M_k=\sum_{\vert j\vert \leq M}D_{k+j}$.  Then, the inverse $T^{-1}_M$ exists and there is a constant $C$ such that for any function $f\in {\M}(x_0,
    r, \beta,\gamma )$ with $x_0\in \R^N, r>0$ and $0<\beta,\gamma
    <\varepsilon$, $T^{-1}_M(f)\in {\M}({x_0},r, \beta,\gamma
    )$  and the following inequality holds:
    \begin{align}
    \Vert T^{-1}_M(f)\Vert _{{{\M}}({x_1},r, \beta,\gamma )}\leq
    C\|f \|_{{{\M}}({x_1},r, \beta,\gamma
        )}.\label{(e3.2.4)}
    \end{align}
\end{prop}
\begin{proof}
By Coifman's decomposition of the identity operator, we express the identity as
\begin{align*}
    I=\sum_k\sum_lD_kD_l= T_M + R_M,
\end{align*}
where $T_M= \sum\limits_{k\in \Z}\sum\limits_{\vert l \vert \leq M}D_{k+l}D_k$ and $R_M= \sum\limits_{\vert l \vert >M}\sum\limits_{k\in \Z}D_{k+l}D_k.$
The crucial estimate is as follows:
    \begin{align}\label{3.1}
    \|R_M(f)\|_{\M(\beta,\gamma,r,x_0)} & \le C 2^{-M\delta}\|f\|_{\M(\beta,\gamma,r,x_0)},
    \end{align}
for some $\delta>0$.

To show \eqref{3.1}, we shall apply Theorem \ref{th2.1}. Let $R_M(x, y)$ be the kernel
of the operator $R_M$. By Lemma \ref{d-difference},
\begin{align*}
    | R_M(x, y)| &=\bigg| \sum_{| l |
        >M}\sum\limits_{k\in \Z}D_{k+l}D_k(x,y) \bigg| \leq \sum_{
        | l | >M}\sum\limits_{k\in \Z}| D_{k+l}D_k(x,y) |\\
 &\lesssim \sum_{| l |
        >M}\sum\limits_{k\in \Z}2^{-| l  | {\varepsilon'}}{{2^{-((k+l)
                \wedge k)\varepsilon } }\over{(2^{-((k+l)\wedge k )} + d(x,y))^{N+\varepsilon }}}
                \lesssim 2^{-M{\varepsilon ^\prime }}d(x,y)^{-N},
    \end{align*}
which ensures that $R_M^1(x, y)$ decays rapidly with respect to the distance $d(x,y)$, fulfilling the required size condition \eqref{1.1}.

To verify the regularity condition \eqref{1.2}, we take the geometric mean of \eqref{e3.1.18} and \eqref{e3.1.19}, thereby obtaining that for $d(x,x')\leq \frac{1}{2} d(x,y),$
\begin{align}\label{(e3.2.12)}
\begin{split}
&| D_{k+l}D_k(x, y) - D_{k+l}D_k(x^\prime , y) | \\
& \quad \leq C2^{-| l | \delta }\left(\frac{d(x,x')}{2^{-((k+l)\vee k)}+d(x,y)}\right)^{\varepsilon'}
\frac{2^{-((k+l)\vee k){\varepsilon' }}}{(2^{-((k+l)\vee k)}+d(x,y))^{N+\varepsilon' }},
\end{split}
\end{align}
where  $0<\varepsilon ^\prime <\varepsilon $ and $\delta >0.$  Consequently, for
    $d(x,x')\leq \frac{1}{2} d(x,y),$ we have
\begin{align*}
        |R_M(x, y) - R_M(x^\prime, y)| &=\bigg| \sum_{|
            l | >M}\sum\limits_{k\in \Z} [D_{k+l}D_k(x,y) - D_{k+l}D_k(x^\prime,
        y)]\bigg| \\
        &\leq C\sum_{| l | >M}\sum\limits_{k\in \Z} 2^{- |l| \delta } \Big(\frac{d(x,x')}{2^{-(k\wedge l )}+d(x,y)}\Big)^{\varepsilon ^\prime} \frac{2^{-(k\wedge l )\varepsilon^\prime }}{(2^{-(k\wedge l )}+d(x,y))^{N+{\varepsilon ^\prime }}}\\
        &\leq C2^{-M\delta }d(x,x')^{\varepsilon'}d(x,
        y)^{-(N+{\varepsilon ^\prime })}.
    \end{align*}
The proof of the smoothness condition \eqref{1.3} follows the same logic.

Similarly, by applying the geometric mean of the estimates provided in Lemma \ref{d-difference}, we can corroborate the double difference condition \eqref{e:double}, as referenced in \cite{DH}.
Moreover, applying the Cotlar-Stein Lemma gives that $R_M$ is bounded on $L_G^2(\R^N)$ with $\|R_M\|_{L^2_G\to L^2_G}\leq 2^{-M{\varepsilon^\prime}}.$
The cancellation conditions $R_M(1)=R_M^*(1)=0$ follows from
$$\int\limits_{\R^N} D_k(x,y)dx=\int\limits_{\R^N} D_k(x,y)dy=0.$$ 

Now, equation \eqref{3.1} follows by invoking Theorem \ref{th2.1}.
This result together with the fact that $T^{-1}_M=\sum\limits
    _{k=0}^\infty (R_M)^k$ if $M$ is sufficiently large, implies (\ref{(e3.2.4)}) and hence Proposition \ref{proposition 3.2.3}.
\end{proof}

We resume the proof of Theorem \ref{th2.2}.
By Proposition \ref{proposition 3.2.3}, we have ${\widetilde
        D}_k(x, y)
    =T^{-1}_M[D^M_k(\cdot, y)](x),$ where the kernel of ${\widetilde D}_k$  belongs to the space  ${\M}(y,2^{ -k},\varepsilon ^\prime,\varepsilon ^\prime )$ with $0<\varepsilon ^\prime <\varepsilon $. The properties $\int\limits_{\R^N} {\widetilde D}_k(x, y) dy=\int\limits_{\R^N} {\widetilde D}_k(x, y) dx= 0$ for all $k\in \Z$ are consequences of the facts that $D^M_k(1)= (T^{-1}_M)^*(1) = 0.$ Similarly, define ${\widetilde{\widetilde
            D}}_k=D^M_kT^{-1}_M.$ Then, ${\widetilde{\widetilde
            D}}_k(x, y) = [D^M_k(x,\cdot)T_M^{-1}](y),$ the kernel of ${\widetilde{\widetilde D}}_k$ is in the space ${\M}(x, 2^{-k},\varepsilon ^\prime,\varepsilon ^\prime )$ with $0<\varepsilon ^\prime
    <\varepsilon $, and the properties $\int\limits_{\R^N}
    {\widetilde{\widetilde D}}_k(x, y) dy=\int\limits_{\R^N}
    {\widetilde{\widetilde D}}_k(x, y) dx= 0, \,k\in
    \mathbb Z,$ follow from the facts that $(D^M_k)^*(1)=T^{-1}_M(1)=
    0.$

All that remains is to establish the convergence of the series \eqref{(e3.2.2)} in  ${\M}(\beta
    ^\prime,\gamma ^\prime )$ for $0<\beta ^\prime <\beta $ and
    $0<\gamma ^\prime
    <\gamma$.

Let us first assume that $f\in {\M}(\beta,\gamma ).$
Then, the convergence of \eqref{(e3.2.2)} in ${\M}(\beta
    ^\prime,\gamma ^\prime )$ is equivalent to
$$\lim\limits_{M\rightarrow \infty } \bigg\Vert \sum
    \limits_{\vert k\vert \leq M}{\widetilde D}_kD_k(f) - f \bigg\Vert
    _{{\M}({\beta ^\prime },{\gamma ^\prime })}= 0$$
for $0<\beta ^\prime <\beta $ and $0<\gamma ^\prime <\gamma.$

Since
\begin{align*}
        \sum_{\vert k\vert \leq M}{\widetilde D}_kD_k(f)
&=T^{-1}_M
        \bigg( \sum_{\vert k\vert \leq M}D^M_kD_k(f)\bigg) =T^{-1}_M \bigg( T_M -
        \sum_{\vert k\vert >M} D^M_k D_k(f)\bigg)\\
        &= f - \lim\limits_{m\rightarrow \infty }R^m_M(f) - T^{-1}_M \bigg( \sum_{\vert k\vert >M}D^M_kD_k(f)\bigg),
    \end{align*}
to demonstrate the convergence of \eqref{(e3.2.2)} in $\M
    (\beta',\gamma') $, it is sufficient to show that
\begin{align}
    \lim\limits_{m\rightarrow \infty } \Vert R^m_N(f) \Vert _{{\M}({\beta ^\prime },{\gamma ^\prime })}= 0 \label{(e3.2.20)}
    \end{align}
and
\begin{align}
    \lim\limits_{M\rightarrow \infty } \bigg\Vert T^{-1}_M \bigg( \sum
    \limits_{\vert k\vert >M}D^M_kD_k(f)\bigg) \bigg\Vert _{{ \M}({\beta
            ^\prime },{\gamma ^\prime })}= 0.\label{(e3.2.21)}
    \end{align}

According to the estimates for $R_M$ and Theorem \ref{1.3},  we have for $0<\beta
    ^\prime <\beta $ and $0<\gamma ^\prime <\gamma,$
$$\Vert R^m_M(f) \Vert _{{\M}({\beta ^\prime },{
            \gamma ^\prime })}\leq (C2^{-M\delta })^m\Vert  f \Vert _{{\M
        }({\beta ^\prime },{\gamma ^\prime })}\leq (C2^{-M\delta })^m\Vert
    f \Vert _{{\M}(\beta ,\gamma )},$$
which confirms (\ref{(e3.2.20)}). The proof of (\ref{(e3.2.21)}) relies on the following estimate
\begin{align}
    \bigg\Vert \sum_{\vert k\vert >M}D^M_k D_k(f) \bigg\Vert _{{\M}({\beta ^\prime },{\gamma ^\prime })} \leq C2^{-\sigma M}\Vert  f
    \Vert _{{\M}( \beta,\gamma )}\label{(e3.2.22)}
    \end{align}
for all $0<\beta ^\prime <\beta $ and $0<\gamma ^\prime <\gamma $, where $\sigma >0$, and the constant $C$ is independent of $f$ and $M$.

Assuming (\ref{(e3.2.22)}) for the moment, by Proposition \ref{proposition 3.2.3}, for $0<\beta ^\prime <\beta $ and 0$<\gamma ^\prime <\gamma,$ we obtain
    \begin{align*}
        \bigg\Vert T^{-1}_M \bigg( \sum_{\vert k\vert >M}D^M_k D_k(f)\bigg) \bigg\Vert
        _{{\M}({\beta ^\prime },{\gamma ^\prime })}&\leq C\bigg\Vert
        \sum_{\vert k\vert >M}D^M_k D_k(f)\bigg\Vert _{{\M}({\beta^\prime },{\gamma ^\prime })} \leq C2^{-\sigma M} \Vert  f  \Vert _{{\M}(\beta,\gamma )},
    \end{align*}
which yields (\ref{(e3.2.21)}).

To show (\ref{(e3.2.22)}), employing a proof analogous to that presented in \cite[pp.~51-53]{DH}  we obtain that for $0<\beta
    ^\prime <\beta $ and $0<\gamma ^\prime <\gamma $, there exists a constant $C$ independent of both $f$ and $M$, along with some $\delta^\prime >0$, such that
    \begin{align}
    \bigg\vert \sum_{\vert k\vert >M}D^M_kD_k(f)(x) \bigg\vert \leq
    C2^{-\delta^\prime M}( 1 +d (x,x_0))^{-(1+{\gamma ^\prime })}\Vert  f
    \Vert _{{\M}(\beta,\gamma )},\label{(e3.2.23)}
    \end{align}
and
    \begin{align}\label{(e3.2.24)}
\begin{split}
    &\bigg\vert\sum_{\vert k\vert >M}D^M_kD_k(f)(x) - \sum
    _{\vert k\vert >M}D^M_kD_k(f)(x^\prime ) \bigg\vert  \\
    & \qquad \leq C \Big({{d (x,x^\prime )}\over{1 + d (x,x_0)}}\Big)^{\beta
        ^{\prime \prime }}( 1 +d (x,x_0))^{-(1+{\gamma ^\prime })}\Vert f
    \Vert _{{\M}(\beta,\gamma )}
\end{split}
    \end{align}
for  $d (x,x^\prime )\leq {{1}\over{2}}( 1 +d (x,x_0))$ and any $0<\beta ^\prime <\beta ^{\prime \prime }<\beta.$

By taking the geometric mean between (\ref{(e3.2.24)}) and the subsequent estimate
\begin{align}\label{(e3.2.25)}
\begin{split}
 &\bigg\vert \sum_{\vert k\vert >M}D^M_kD_k(f)(x) - \sum
    \limits_{\vert k\vert >M}D^M_kD_k(f)(x^\prime )
    \bigg\vert \\
    &\qquad\leq \bigg\vert \sum_{\vert k\vert >M}D^M_kD_k(f)(x) \bigg\vert
    +\bigg\vert \sum_{\vert k\vert >M}D^M_kD_k(f)(x^ \prime )\bigg\vert \\
    & \qquad \leq C2^{-\delta^\prime M }( 1 +d (x,x_0))^{-(1+{\gamma ^\prime
        })}\Vert f \Vert _{{\M}(\beta,\gamma )}
\end{split}
    \end{align}
    for $d (x,x^\prime )\leq {{1}\over{2}}( 1 +d
    (x,x_0)),$ we derive the following inequality:
\begin{align}\label{(e3.2.26)}
\begin{split}
    & \bigg\vert \sum_{\vert k\vert >M}D^M_kD_k(f)(x) - \sum
    \limits_{\vert k\vert >M}D^M_kD_k(f)(x^\prime ) \bigg\vert \\
    & \qquad \leq C2^{-M{\delta}}\Big({{d (x,x^\prime )}\over{1 + d
            (x,x_0)}}\Big)^{\beta ^\prime }( 1 +d (x,x_0))^{-(1+{\gamma ^\prime
        })}\Vert  f  \Vert _{{\M}(\beta,\gamma
        )}
\end{split}
    \end{align}
    for $d (x,x^\prime )\leq {{1}\over{2}}( 1 +d (x,x_0)).$

Incorporating  (\ref{(e3.2.23)}) and (\ref{(e3.2.26)}), along with the fact that
    \begin{align*}
    \int\limits_{\R^N} \sum_{\vert k\vert >M}D^M_kD_k(f)(x) dx
    = \int\limits_{\R^N} \sum_{\vert k\vert >M}D_k(f)(y)(D^M_k)^*(1)(y)dy= 0,
    \end{align*}
we deduce that
\begin{align}
    \sum_{\vert k\vert >M}D^M_kD_k(f)(x) dx
    \in {\M}(\beta^\prime, \gamma^\prime)\nonumber
    \end{align}
and that
\begin{align}
    \bigg\Vert \sum \limits_{\vert k\vert >M}D^M_kD_k(f) \bigg\Vert
    _{{\M}({\beta ^\prime },{\gamma ^\prime })}\leq
    C2^{-\delta^\prime M} \Vert  f  \Vert _{{\M}(\beta ,\gamma
        )},\nonumber
    \end{align}
which confirms (\ref{(e3.2.22)}). This concludes the proof of Theorem \ref{th2.2}.
\end{proof}

\section{Proof of theorem \ref{th1.1}}\label{sec:3}
We first present the definition of Besov spaces as follows.
\begin{definition}\label{def3.1}
    Let $f \in (\M(1,1,r,x_0))^\prime$ and $0<\alpha<1.$ The Besov space $\binfty$ is defined to consist all $f$ satisfying
    $$\|f\|_{\binfty}=\sup\limits_{\substack{k\in \Bbb Z \\ x \in \R^N}}2^{\alpha k}|D_kf(x)|<\infty.$$
    Similarly, the Besov space $\b1 $ comprises all $f$ such that
$$
    \|f\|_{\b1 }=\sum\limits_{k=-\infty}^\infty \int\limits_{\R^N} 2^{-\alpha k}|{\widetilde D}_k(f)(x)|dx<\infty,$$
    where $D_k$ and ${\widetilde D}_k$ are given as in Theorem \ref{th2.2}.
\end{definition}
It is important to note that these spaces are well-defined, meaning they are independent of the choice of approximations $S_k$ (and corresponding $D_k$).
Furthermore, for $0<\alpha<1$, the dual space of $\b1 $ is $\binfty$; see~\cite{HMY}. The following result characterizes the H\"older space:
\begin{pro}\label{pro3.1}
    A function $f$ belongs to the H\"older space ${{C}_G^\alpha}(\R^N)$ with $0<\alpha<1$ if and only if $f$ is in $\binfty$. Moreover,  $\|f\|_{{C}_G^\alpha}\sim \|f\|_{\binfty}.$
\end{pro}

\begin{proof}
Let $f$ be an element of ${{C}}_G^\alpha(\R^N)$ and $g$ be an element of $\M(\beta,\gamma,r,x_0)$ with  $0<\alpha<\gamma\le 1.$ Then, by virtue of the cancellation and size conditions imposed on $g$, we have
$$\Big|\int\limits_{\R^N}f(x)g(x)dx\Big|=\Big|\int\limits_{\R^N}[f(x)-f(x_0)]g(x)dx\Big|\le C\|f\|_{{{C}}_G^\alpha}\|g\|_{\M(\beta,\gamma,r,x_0)},$$
which implies that $f\in \binfty.$

In particular, $$|D_k(f)(x)|=\Big|\int\limits_{\R^N}D_k(x,y)[f(y)-f(x)]dy\Big|\le C2^{-\alpha k}\|f\|_{{C}_G^\alpha}$$
which implies that $\|f\|_{\binfty}\lesssim \| f\|_{\binfty}.$

We wish to emphasize that for $f$ belonging to ${{C}}_G^\alpha(\R^N)$ with $0<\alpha<\gamma\leq 1$, one can define a distributional pairing  $\langle f,g\rangle$ for any $ g\in \M(\beta,\gamma,r,x_0),$ through the bilinear form
$$\langle f,g\rangle =\sum\limits_{k=-\infty}^\infty \int\limits_{\R^N}{\widetilde D}_k(g)(x)D_k(f)(x)dx.$$
Indeed, if $g\in \M(\beta,\gamma,r,x_0),$ then $g\in \b1 (\R^N)$ and
$$\sum\limits_{k=-\infty}^\infty \int\limits_{\R^N}|{\widetilde D}_k(g)(x)D_k(f)(x)|dx\lesssim \|f\|_{\binfty}\|g\|_{\dot{B}_1^{-\alpha, 1}}.$$

To demonstrate that the inclusion $f\in \binfty$ implies $f\in {{C}}_G^\alpha,$ consider $f$ an element of $f\in \binfty$ and $g\in \M(\beta,\gamma,r,x_0)$ with $0<\alpha<\gamma.$  Utilizing the Calder\'on reproducing formula as stated in Theorem \ref{th2.2}, we can express
\begin{align*}\langle f,g\rangle =\bigg\langle \sum\limits_{k=-\infty}^0 {\widetilde{D}}_kD_k(f), g\bigg\rangle
    +\bigg\langle \sum\limits_{k=1}^\infty {\widetilde{D}}_k(g)D_k(f), g\bigg\rangle
    =:I+I\!I.
\end{align*}
Thanks to the cancellation property of $g$, we can write
$$I=\left\langle \sum\limits_{k=-\infty}^0 \int\limits_{\R^N}[{\widetilde{D}}_k(\cdot,y)-{\widetilde{D}}_k(x_0,y)]D_k(f)(y)dy, g \right\rangle .$$
Now, applying the smoothness condition of ${\widetilde{D}}_k$ and the fact that  $f\in \binfty$, we conclude that for any $0<\alpha<\varepsilon$,
$$\int\limits_{\R^N}\int\limits_{\R^N}\bigg|{\widetilde{D}}_k(x,y)-{\widetilde{D}}_k(x_0,y)
\bigg|\big|D_k(f)(y)\big|dy\lesssim {d(x,x_0)}^\varepsilon{2^{-k}}^{(\varepsilon-\alpha)}.$$

Consequently, the series $\sum\limits_{k=-\infty}^0 \int\limits_{\R^N}[{\widetilde{D}}_k(x,y)-{\widetilde{D}}_k(x_0,y)]D_k(f)(y)dy$
 converges uniformly for all $x$ with $d(x,x_0)\le M$, where $M>0$ is a fixed constant.
 This implies that the series converges to a continuous function on $\R^N.$

For the term $I\!I,$ the size estimate $|D_k(f)(x)|\lesssim
2^{-\alpha k}$  ensures that  $\sum\limits_{k=1}^\infty {\widetilde{D}}_kD_k(f)(x)$  converges uniformly on $\R^N$,
 thus it defines a continuous function for all $x$. Hence, $f$ is confirmed to be a continuous function.

To verify that $f\in
{{C}_G^\alpha}(\R^N),$ for any $x, x'\in \R^N$ we choose some $k_0$ such that $2^{-k_0}\le d(x,x')<2^{1-k_0}.$ We split $f(x)-f(x')$ into two parts: $A=\sum\limits_{k\le k_0}\int\limits_{\R^N
    }[{\widetilde{D}}_k(x,y)-{\widetilde{D}}_k(x',y)]D_k(f)(y)dy$ and $B=\sum\limits_{k>k_0}\int\limits_{\R^N
    }[{\widetilde{D}}_k(x,y)-{\widetilde{D}}_k(x',y)]D_k(f)(y)dy.$

For the first series $A$, applying the smoothness condition of ${\widetilde{D}_k}$ and the estimate of $D_k(f)$ yields
$$|A|\lesssim \|f\|_{\binfty}\sum\limits_{k\le k_0} 2^{(\varepsilon-\alpha) k
}d(x,x')^\varepsilon \lesssim d(x,x')^\alpha \|f\|_{\binfty}.$$
The estimates of $D_k(f)$ and the size condition of ${\widetilde{D}_k}$ give
$$|B|\lesssim \sum\limits_{k>k_0}
2^{-\alpha k}\|f\|_{\binfty}\lesssim d(x,x')^\alpha \|f\|_{\binfty}.$$
Hence, $f\in
{{C}_G^\alpha}(\R^N)$ with $\|f\|_{{C}_G^\alpha}\le C\|f\|_{\binfty}$.
\end{proof}

We are ready to show Theorem \ref{1.1}.

\begin{proof}
We now proceed to demonstrate Theorem \ref{1.1}. The proof is derived from the following estimate: there exists a constant $C$ such that for any $f\in
{C}_G^\alpha(\R^N)$ with $0<\alpha<\varepsilon,$ where $\varepsilon$ denotes the regularity exponent of the kernel of $T$, the inequality
\begin{equation}\label{ineq3.1}
|\langle D_k, Tf\rangle |\lesssim 2^{-k\alpha}\|f\|_{C_G^\alpha}
\end{equation}
holds, where $D_k(x, y)$ are the functions defined in Theorem \ref{th2.2} for any integer $k$.

To show the estimate in \eqref{ineq3.1}, assume that $f(x)\in {{C}}_G^\alpha(\R^N)$ with $0<\alpha<\varepsilon\le 1.$ Define $1=\xi(z)+\eta(z)$, where $\xi(z)=\theta\Big(\frac{d(z,x)}{2^{5-k}}\Big)$ and $\theta$ is a smooth function with compact support in $\mathbb{R}$, satisfying $\theta(x)=1$ for $|x|\le 1$ and $\theta(x)=0$ for $|x|\ge 2$. Given that $T(1)=0$ and $D_k(x, \cdot)\in {{C}}_{G,0,0}^1(\R^N)$ for any fixed integer $k$ and point $x\in \R^N$, we may express
\begin{align*}
&\langle D_k, Tf\rangle =\langle K(y,z),f(z)D_k(x,y)\rangle
= \langle K(y,z),[f(z)-f(x)] D_k(x,y)\rangle\\
&=\langle K(y,z),(\xi(z)+\eta(z))[f(z)-f(x)]D_k(x,y)\rangle
\\&=\langle K(y,z),\xi(z)(f(z)-f(y))D_k(x,y)\rangle +\langle
K(y,z),\xi(z)[f(y)-f(x)]D_k(x,y)\rangle
\\&\qquad +\langle K(y,z),\eta(z)[f(z)-f(x)]D_k(x,y)\rangle
\\&=:  I_1+I_2+I_3.
\end{align*}

For the initial term $I_1$, we have
\begin{align*}
|I_{1}|&\le   \iint\limits_{{d(y,z)\le
        20\cdot 2^{-k}}} \Big|K(y,z)\xi(z)[f(z)-f(y)] D_k(x,y)\Big|dydz \\
&\lesssim \iint\limits_{d(y,z)\le 20\cdot 2^{-k}} \frac{1}{d(y,z)^N} d(y,z)^\alpha|D_k(x,y)|
dzdy\cdot\|f\|_{C_G^\alpha}\\
&\lesssim 2^{-k\alpha}\|f\|_{C_G^\alpha}.
\end{align*}

Employing the same argument used for estimating the term $I_2$ in Theorem \ref{th1.1}, it follows that
\begin{equation}\label{3.2}
|T\xi(x)|\lesssim 1
\end{equation}
and thus,
$$|I_2|\lesssim \|[f(x)-f(\cdot)]D_k(x,\cdot)\|_1\lesssim 2^{-k\alpha}\|f\|_{C_G^\alpha}.$$

To estimate the term $I_3$, due to the cancellation property of $D_k(x,y)$, we have
$$\langle K(y,z),\eta(z)[f(z)-f(x)]D_k(x,y)\rangle
=\langle {K(y,z)-K(x,z)},\eta(z)[f(z)-f(x)]D_k(x,y)\rangle.$$
Noting that by the support conditions of $D_k(x,y)$ and $\eta(z)$, we are dealing with the case where $d(x,y) \le
5\cdot 2^{-k}$ and $d(z,x)\ge
20\cdot 2^{-k}$. This implies that $d(z,x)\ge \frac{1}{2}d(x,y)$. Consequently, invoking the smoothness of $K(x,z)$ and the H\"older continuity of $f\in {C}_G^\alpha$ with $0<\alpha<\varepsilon$, we obtain
\begin{align*}|I_3|
&\lesssim  \int\limits_{d(z,x)\ge 20\cdot 2^{-k}\ge \frac{1}{2}d(x,y)}
\frac{1}{d(x,z)^N}\Big(\frac{d(x,y)}{d(x,z)}\Big)^\varepsilon
d(x,z)^\alpha |D_k(x,y)|dzdy\|f\|_{C_G^\alpha}
\\&\lesssim  2^{-k\alpha}\|f\|_{C_G^\alpha},
\end{align*}
establishing \eqref{ineq3.1}.

To conclude the proof of Theorem \ref{1.1}, we note that for $f\in
{{C}}_G^\alpha$ and $g\in \M(\beta,\gamma,r,x_0),$ as previously discussed, the pairing $\langle f,g\rangle $ can be defined by
$$\langle f,g\rangle =\sum\limits_{k=-\infty}^\infty \int\limits_{\R^N }{\widetilde D}_k(g)(x)D_k(f)(x)dx.$$
This leads to
$$\langle Tf,g\rangle =\sum\limits_{k=-\infty}^\infty \int\limits_{\R^N }{\widetilde D}_k(g)(x)D_k(Tf)(x).$$
To confirm that this is well-defined, we apply \eqref{ineq3.1} to obtain
$|D_k(Tf)(x)|\lesssim 2^{-\alpha k}\|f\|_{C_G^\alpha}.$
Consequently,
$|\langle
Tf,g\rangle |\lesssim
\|f\|_{{\dot{C}}_G^\alpha}\|g\|_{{\dot{B}}_{1}^{-\alpha,1}}\lesssim
\|f\|_{{\dot{C}}_G^\alpha}\|g\|_{\M(\beta,\gamma,r,x_0)}.$
Therefore, $\langle Tf,g\rangle $ is indeed well defined, and we can infer that  $T(f)\in \M(\beta,\gamma,r,x_0)^\prime.$ Furthermore, by \eqref{ineq3.1} and the characterization presented in Proposition \ref{pro3.1}, we see that $Tf\in C_G^\alpha$ with
$\|Tf\|_{{{C}}_G^\alpha}\lesssim \|f\|_{{{C}}_G^\alpha}.$
This concludes the proof of Theorem \ref{1.1}.
\end{proof}

\section{Proof of theorems \ref{th1.2} and \ref{th1.3}} \label{sec:4}

We first show Theorem \ref{th1.3}.
\begin{proof}[{\bf The proof of Theorem \ref{th1.3}}]
To begin with our analysis, we initially present the weak type $(1,1)$ estimate. The fundamental approach involves the utilization of the classical Calder\'on-Zygmund decomposition. For this purpose, let $f\in L_G^2(\R^N)\cap L_G^1(\R^N)$ and $\lambda$ a positive real number. We define $E_\lambda=\{x:M(f)(x)>\lambda\},$ with $M(f)(x)=\sup\limits_{B\ni x}\frac1{|B|}\int\limits_{B}|f(y)|dy$ being the uncentered Hardy-Littlewood maximal function.
It is important to recognize that $E_\lambda$ is $G$-invariant; that is, $\sigma(E_\lambda)=E_\lambda,$  for every $\sigma \in G$, owing to the $G$-invariance of $f$. By adopting the Whitney decomposition for $E_\lambda$, we can express
$E_\lambda= \bigcup\limits_{j=1}^\infty Q_j=\bigcup\limits_{j=1}^\infty \sigma(Q_j),$ for all $\sigma\in G.$
We then define the functions $b_j^\sigma(x)$, $g^\sigma(x)$ for each $\sigma \in G$ by
$b_j^\sigma(x)=[f(x)-\frac1{|\sigma(Q_j)|}\int\limits_{\sigma(Q_j)}f(x)dx] \chi_{\sigma(Q_j)}(x),$
 $g^\sigma(x)=f(x)-\sum\limits_{j=1}^\infty b_j^\sigma(x),$
where $\chi_{\sigma(Q_j)}$ is the characteristic function of $\sigma(Q_j)$.

To capture the collective contributions of all group elements in $G$, we introduce the averaged functions $g(x)$ and $b_j(x)$ as
$g(x)=\frac1{|G|}\sum\limits_{\sigma\in G}g^\sigma(x),$
$b_j(x)=\frac1{|G|}\sum\limits_{\sigma\in G}b_j^\sigma(x).$
It becomes evident that there exists a constant $C > 0$ such that
\\
(i)
$f(x)=g(x)+\sum\limits_{j=1}^\infty b_j(x), x\in \R^N;$\\
(ii)
$|g(x)|=|f(x)|\le  \lambda,$ for\ almost $ x\notin
E_\lambda;$\\ (iii) $|g(x)|\le
C\lambda,$ for all $ x\in E_\lambda;$\\ (iv)
$|\bigcup\limits_{j=1}^\infty Q_j|\le  C\lambda^{-1}\|f\|_1;$\\
(v) $\|g\|_2\le  C \lambda^{1/2}\|f\|^{1/2}_1;$\\ (vi)
$\int\limits_{\R^N} |b_j(x)|dx\le  C\lambda|Q_j|;$\\ (vii)
$\int\limits_{\R^N} b_j(x) dx=0.$
\\ (viii)
$g(x)=g(\sigma(x)), \forall \sigma\in G.$

Let $\mathcal O(\bigcup\limits\limits_{j=1}^\infty
4\sqrt{N}Q_j)=\{x: d(x,x_{Q_j})\le  \ell(4\sqrt{N}Q_j),
\text{for each}\ j\}$ with $x_{Q_j},$ the center of $Q_j$, and $\ell(4\sqrt{N}Q_j)$ the side length of $4\sqrt{N}Q_j.$  Then,
$$|\mathcal O(\bigcup\limits_{j=1}^\infty 4\sqrt{N}Q_j)|\lesssim
\sum\limits_{j=1}^\infty|Q_j|\lesssim \lambda^{-1}\|f\|_1$$
and
$$\lambda|\big\{x\in \big(\mathcal O(\bigcup\limits_{j=1}^\infty
4\sqrt{N}Q_j)\big)^c: |Tb(x)|\ge  \lambda\big\}|\le
\int\limits_{\big(\mathcal O(\bigcup\limits_{j=1}^\infty
4\sqrt{N}Q_j)\big)^c} |Tb(x)|dx.$$
We estimate the last term as follows. For $x\in \big(\mathcal O(\bigcup\limits_{j=1}^\infty
4\sqrt{N}Q_j)\big)^c,$
$$Tb_j(x)=\int\limits_{\R^N} K(x,y)b_j(y)dy=\int\limits_{{\mathcal O}(Q_j)}[K(x,y)-K(x,x_{Q_j})]b_j(y)dy,$$
where the cancellation condition of $b_j$ is utilized.

Observe that if $x\in \big(\mathcal O(\bigcup\limits_{j=1}^\infty
4\sqrt{N}Q_j)\big)^c, y\in {\mathcal O}(Q_j)$ then
$d(x,y)\ge  2 d(y,x_{Q_j})$ and hence, applying the smoothness
condition of $K(x,y)$ gives
\begin{align*}
\int\limits_{\big(\mathcal O(\bigcup\limits_{j=1}^\infty
4\sqrt{N}Q_j)\big)^c}|Tb_j(x)|dx \lesssim \int\limits_{{\mathcal O}(Q_j)}
\int\limits_{d(x,y)\ge  2 d(y,x_{Q_j})}\frac{d(y,
x_{Q_j})^\varepsilon}{d(x,y)^{N+\varepsilon}} dx |b_j(y)|dy\lesssim
\lambda|Q_j|,
\end{align*}
where we apply the property(vi) and hence, by (iv),
$$\int\limits_{\big(\mathcal O(\bigcup\limits_{j=1}^\infty 4\sqrt{N}Q_j)\big)^c}|Tb(x)|dx\lesssim \sum\limits_{j=1}^\infty\lambda|Q_j|\lesssim \|f\|_1.$$

All these estimates, coupled with the $L_G^2$ boundedness of $T$ and (v), imply that
\begin{align*}
&\Big|\Big\{x: |Tf(x)|\ge  \lambda\Big\}\Big| \lesssim \Big|\Big\{x: |Tg(x)|\ge  {\lambda\over2}\Big\}\Big| + \Big|\Big\{x: |Tb(x)|\ge  {\lambda\over2}\Big\}\Big|\\
& \lesssim  \Big(\frac{\|g\|_2}{\lambda}\Big)^2 +\Big|\mathcal
O(\bigcup\limits_{j=1}^\infty 4\sqrt{N}Q_j)\Big| +\lambda^{-1}
\int\limits_{\big(\mathcal O(\bigcup\limits_{j=1}^\infty 4\sqrt{N}Q_j)\big)^c}|Tb(x)|dx
\lesssim \lambda^{-1}\|f\|_1,
\end{align*}
which implies that $T$ is of  weak type $(1,1)$ for $L_G^1(\R^N).$

By interpolation, we establish that $T$ is bounded on $L_G^p$ for $1<p\le 2$. The same line of reasoning is applicable to $T^*$, leading to its boundedness in $L_G^p$ for the same range of $p$, and subsequently, by duality, we infer the boundedness of $T$ in $L_G^p$ when $2\le p<\infty$.

We proceed to show the $L_G^\infty \to {\rm BMO}_G(\R^N)$ boundedness of $T$. In doing so, We first provide a rigorous definition of $Tf(x)$ for $f\in L^\infty_G.$ To this end, we follow the idea given in \cite{MC}. For any $f\in L^\infty_G(\R^N)$, we introduce a sequence of functions $\{f_j\}$, defined by $f_j(x) = f(x)$ if $|x|\le j$ and $f_j(x) = 0$ otherwise.

Since $f_j\in
L^2_G(\R^N),$  the $L^2_G$ boundedness of $T$ implies that $T(f_j)$ are well-defined $L^2_G$ functions. We assert the existence of a sequence $\{c_j\}_j$ of constants such that $T(f_j)-c_j$ uniformly converges on any compact subset of $\R^N$ to a function in ${\rm BMO}_G(\R^N)$. This function will be defined as $T(f)$, modulo the constant functions. Specifically, let
$c_j=\int\limits_{1\le d(0,y)\le j}K(0,y)f(y)dy.$
Observe that, due to the size condition of $K(x,y)$, we have
\begin{align*}
|c_j|\lesssim \int\limits_{1\le |y|\le j}\frac{1}{|y|^N}dy\lesssim j^N<\infty.
\end{align*}

To show the uniform convergence of $T(f_j)-c_j$ on the compact ball $B(0,R)=\{x: |x|<R \}=\{x: d(x,0)<R\},$ we partition $f_j$ into the sum of two functions, $g$ and $h_j$. Here, $g(x)=f(x)$ for $|x|\le 2R$, and $g(x)=0$ if $|x|>2R$. For $j>2R$, we have that for $|x|\le R$,
\begin{align*}
T(f_j)(x)&=T(g)(x)+T(h_j)(x)=T(g)(x)+\int\limits_{2R\le |y|\le
    j}K(x,y)f(y)dy\\
&= T(g)(x)+\int\limits_{2R\le |y|\le
    j}[K(x,y)-K(0,y)]f(y)dy+c_j-C(R),
\end{align*}
where $C(R)=
\int\limits_{1\le |y|\le 2R}K(0,y)f(y)dy.$ Observe that when $|x|\le R$, by the smoothness condition on the kernel $K(x,y)$, we obtain
\begin{align*}
\int\limits_{2R\le |y|}|K(x,y)-K(0,y)|\cdot|f(y)|dy&\le
\int\limits_{2d(x,0)\le d(y,0)}|K(x,y)-K(0,y)|\cdot|f(y)|dy\\
&\lesssim \int\limits_{2|x|\le
    |y|}\frac{|x|^\varepsilon}{|y|^{N+\varepsilon}}
dy\|f\|_{\infty}\lesssim \|f\|_{\infty}.
\end{align*}
Thus, the integral $\int\limits_{2R\le |y|\le
    j}|K(x,y)-K(0,y)|\cdot|f(y)|dy$ converges uniformly on $|x|\le R$ as $j$ tends to $\infty$, which implies that $T(f_j)-c_j$ converges uniformly on any compact set in $\R^N.$

Now, for $|x|<R,$ let us define $F^R(x)=T(g)(x)+\int\limits_{2R\le |y|}[K(x,y)-K(0,y)]f(y)dy-C(R).$ It is straightforward to verify that for any $0<R_1<R_2,$, when $|x|<R_1,$ we have $F^{R_2}(x)-F^{R_1}(x)=0.$
Hence, we may define $Tf(x)=F^R(x)$ for all $|x|\le R$, where $R>0$. We remark that the smoothness condition on the kernel $K(x,y)$ can be replaced by
$$\int\limits_{d(x,y)\ge 2d(x, x')} |K(x,y)-K(x',y)|dy\le C.$$

For any function $f\in L^\infty_G(\R^N)$, we can show that  $T(f)\in {\rm BMO}_G(\R^N)$ and moreover,
$\|T(f)\|_{{\rm BMO}}\lesssim \|f\|_{\infty}.$
To prove this, consider an arbitrary ball  $B=\{y: |x_0-y|\le R\}$ with center $x_0$ and radius $R$. Let $\mathcal O_{B_1}=\{y:
d(x_0,y)\le 2R\}.$ We decompose $f$ into two parts, $f=f_1+f_2$, where $f_1$ is the product of $f$ with the characteristic function of $\mathcal{O}_{B_1}$. It is evident that $f_1\in L^2_G(\R^N)$. Therefore, $T(f_1)$ is well-defined.
Note that $\mathcal O_B\subseteq B(0,R_1),$ where $R_1=|x_0|+R.$ Consequently, for any $x\in B,$
\begin{align*}
T(f_2)(x)&=T(f)(x)-T(f_1)(x)\\
&=T(f(\cdot)\chi_{\{|\cdot|\le 2R_1\}})(x)+\int\limits_{2R_1\le |y|}[K(x,y)-K(0,y)]f(y)dy-C(R_1)-T(f_1)(x)\\
&=\int\limits_{|y|\le 2R_1\atop d(x_0,y)>2R}K(x,y)f(y)dy+\int\limits_{2R_1\le |y|}[K(x,y)-K(0,y)]f(y)dy-C(R_1)\\
&=\int\limits_{|y|\le 2R_1\atop d(x_0,y)>2R}[K(x,y)-K(x_0,y)]f(y)dy+\int\limits_{2R_1\le |y|}[K(x,y)-K(0,y)]f(y)dy-C(R_1)+D,
\end{align*}
where $D=\int\limits_{|y|\le 2R_1\atop d(x_0,y)>2R}K(x_0,y)f(y)dy.$

Hence, due to the smoothness property of $K,$ for any $x\in B,$ we obtain
\begin{align*}
\int\limits_{|y|\le 2R_1\atop d(x_0,y)>2R}|K(x,y)-K(x_0,y)|\cdot|f(y)|dy&\le \int\limits_{d(x_0,y)\ge 2d(x_0,x)}|K(x,y)-K(x_0,y)|\cdot|f(y)|dy\\
&\lesssim \|f\|_\infty \int\limits_{d(x_0,y)\ge
2d(x_0, x)}\frac{d(x_0,x)^\varepsilon }{d(x_0,y)^{N+\varepsilon}}dy\lesssim \|f\|_\infty
\end{align*}
and
\begin{align*}
\int\limits_{2R_1\le |y|}|K(x,y)-K(0,y)||f(y)|dy&\le \int\limits_{|y|\ge 2|x|}|K(x,y)-K(0,y)|\cdot|f(y)|dy\\
&\lesssim \|f\|_\infty \int\limits_{|y|\ge 2|x|}\frac{|x|^\varepsilon }{|y|^{N+\varepsilon}}dy\lesssim \|f\|_\infty,
\end{align*}
which implies
\begin{align*}
\Big(\int\limits_{B}\Big|T(f_2)(x)-\frac{1}{|B|}\int\limits_{B} T(f_2)(y)dy\Big|^2
dx\Big)^{1/2}&\le C|B|^{1/2} \|f\|_\infty.
\end{align*}
Additionally, we have the inequality
$$\|T(f_1)\|_2\le \|T\|_{L^2_G\to L^2_G}\|f_1\|_2\lesssim \|f\|_\infty |B|^{1/2}.$$
Hence, we can derive
\begin{align*}
\Big(\int\limits_{B}\Big|T(f)(x)-\frac{1}{|B|}\int\limits_{B} T(f)(y)dy\Big|^2
dx\Big)^{1/2}
&\le
C|B|^{1/2}\|f\|_\infty,
\end{align*}
and thus the proof of Theorem \ref{th1.3} is concluded.
\end{proof}

Applying Theorems \ref{th1.1} and \ref{th1.3}, we finally present the
\begin{proof}[\bf Proof of Theorem \ref{th1.2}, the $T1$ Theorem]

We begin by establishing the necessity part of Theorem \ref{th1.2}.
Suppose that $T$ is a singular integral operator as defined in Definition \ref{def1.1}, which is bounded on the space $L_G^2(\R^N).$ We deduce from Theorem \ref{th1.3} that $T(1)\in {{\rm BMO}_G}(\R^N)$ and $T^*(1)\in {\rm BMO_G}(\R^N).$ To verify that $T\in WBP_G$, consider $f, g\in C_{G,0}^\eta(\R^N),\eta>0$ with $\supp (f)\subset \mathcal{O}_B(x_0, r)$ and $\supp (g)\subset \mathcal{O}_B(y_0, r)$ for $x_0, y_0\in \mathbb{R}^N$ and $r>0$, $\|f\|_\infty \leq 1$, $\|g\|_\infty \leq 1$ and the respective $C_{G,0}^\eta$ norms of $f$ and $g$ are bounded by $r^{-\eta}$. A straightforward computation yields $\|f\|_2+\|g\|_2\lesssim r^{N/2}$. This, together with the $L^2_G$ boundedness of $T$, implies
$$|\langle T (f), g\rangle|\le \|T(f)\|_2 \|g\|_2\lesssim \|f\|_2 \|g\|_2 \lesssim r^{N},$$
and hence $T\in WBP_G.$

The proof of the sufficiency part is crucial.
Our argument begins with the extension of $T$ to a continuous linear operator from $\Lambda_G^s\cap L_G^2(\R^N)$  into $(C_{G,0}^s(\R^N))^\prime$, where $\Lambda_G^s(\Bbb
R^N)$ represents the closure of $C_{G,0}^\eta(R^N)$ with respect to the norm $\|\cdot\|_{G,s}$, for $0 < s < \eta$. To proceed with precision, let us consider $g\in C_{G,0}^s(\R^N),$$0 < s< \eta < 1$,  with  supp~$g\subset \mathcal{O}_B(x_0, r)$. We select $\theta\in C_{G,0}^s(\R^N)$ such that $\theta(x) = 1$ for $d(x, x_0) \leq 2r$ and $\theta(x) = 0$ for $d(x, x_0) \geq 4r$. For any $f\in \Lambda_G^s \cap
L_G^2(\R^N),$ we can express
$$\langle Tf, g \rangle =\langle T(\theta f), g\rangle
+\langle  T((1-\theta) f), g \rangle. $$
The first term is well-defined since $\theta f\in C_{G,0}^s(\R^N).$ To see the second term is also well-defined, we use the size condition of $K(x,y)$ and the fact that $f\in L_G^2(\R^N)$ to write
$$\langle T((1-\theta) f), g \rangle
= \int_{\R^N} g(x)\Big(\int\limits_{\{y:d(x,y)>r\}} K(x,y)(1-\theta(y))f(y)dy\Big)\ dx.
$$
By H\"older's  inequality, we have
\begin{align*}
\int\limits_{\{y:d(x,y)>r\}} \bigg| K(x,y)(1-\theta(y))f(y)\bigg|dy
\lesssim \|f\|_{L^2_G(\R^N)}\Big( \sum\limits_{j=0}^\infty
2^{-jN}r^{-N}\Big)^{\frac12} <\infty.
\end{align*}
This implies that $\langle  T((1-\theta) f), g \rangle $ is indeed well-defined. Furthermore, this extension is independent of the particular choice of $\theta$.

Recall from Section~\ref{sec:2} that the identity operator $I$ on $L_G^2(\R^N)$ can be written as
$$I=\sum\limits_{k=-\infty}^\infty {D}_k=\sum\limits_{k=-\infty}^\infty\sum\limits_{j=-\infty}^\infty {D}_k{D}_j={T}_{M} +{R}_{M},$$
where ${T}_{M}=\sum\limits_{k\in \Bbb Z} D_kD_k^{M}$, with $D_k^M=\sum\limits_{\{j\in \Bbb Z:|j|\le M\}} D_{k+j}$, and
${R}_{M}=\sum\limits_{\{j,k\in \Bbb Z:\,
    |k-j|> {M}\}}{D}_k{D}_j.$
We know that, for a sufficiently large integer $M$, the inverse of $T_M$, denoted $T_M^{-1}$, is bounded on $L_G^2(\R^N)$. Furthermore, we have
$$I=T_M^{-1}T_M=T_M^{-1}\sum\limits_{k=-\infty}^\infty D_k^{M}D_k.$$
Consequently,
$$\|f\|_{L_G^2}\sim \|T_M(f)\|_{L_G^2}.$$

The following lemma describes the properties of operators ${T}_{M}$ when acting on $\Lambda_G^s.$

\begin{lemma}\label{le4.1}
    Suppose $0<s<\frac12 .$ Then
    {\bf (i)} ${T}_{M}=\sum\limits_{k=-\infty}^\infty D_kD^{M}_k$ converges in the norm of $\Lambda_G^s$ uniformly;
    {\bf (ii)} ${T}_{M}$ is bounded on $\Lambda_G^s$; {\bf (iii)} $\|{T}_{M}-I\|_{G,s}\to 0$ as $M\to +\infty$.
\end{lemma}
We postpone the proof of this lemma to the Appendix.

We are now ready to present the verification of the sufficiency part of Theorem \ref{th1.2} under the assumptions that $T(1)(x)= T^*(1)(x)=0$ and $T\in WBP_G.$ Noticing that
${T}_{M}$ converges strongly on $L_G^2(\R^N)$ and applying the almost orthogonal estimates alongside the Cotlar--Stein Lemma, we deduce the following:
$$\sup\limits_{L_1,L_2}\Big\|\sum\limits_{k=L_1}^{L_2}
D_kD^{M}_k\Big\|_{L_G^2(\R^N)\to
L_G^2(\R^N)}<+\infty.$$
Hence, by virtue of Lemma \ref{le4.1}, ${T}_{M}$ converges strongly on $\Lambda_G^s\cap L_G^2(\R^N)$.

To show the sufficiency component of Theorem \ref{1.3}, we first establish that
$$\big|\langle g_0, Tf_0 \rangle\big|\lesssim \|g_0\|_{L_G^2(\R^N)}\|f_0\|_{L_G^2(\R^N)}$$
for any $g_0, f_0\in \Lambda_G^s\cap L_G^2(\R^N)$ with compact supports. Given  $f_0\in \Lambda_G^s\cap L_G^2(\R^N)$ with compact support, set
$f_1={T}_{M}^{-1}f_0\in \Lambda_G^s\cap L_G^2(\R^N) $  and
$$U_{L_1,L_2}
=\sum\limits_{k=L_1}^{L_2}D_kD^{M}_k.$$
It follows from Lemma \ref{le4.1} that
$\lim\limits_{\overset{L_1\to -\infty}{L_2\to +\infty}}
U_{L_1,L_2}f_1=f_0$ within the space $\Lambda_G^s\cap L_G^2(\R^N)$. Observe that $T$ can be extended to a continuous linear operator from $\Lambda_G^s \cap L_G^2(\R^N)$ into $(C^s_0(\R^N))^\prime.$
It follows that, for each $g_0\in \Lambda_G^s\cap L_G^2(\R^N)$ with compact support,
$$\langle g_0, Tf_0 \rangle =\lim\limits_{\overset{L_1\to -\infty}{L_2\to +\infty}}\langle g_0, TU_{L_1,L_2}f_1 \rangle.$$
Similarly, if we let $g_1={T}_{M}^{-1}g_0,$ then $g_1\in \Lambda_G^s\cap L_G^2(\R^N)$ and $U_{L'_1,L'_2}g_1$ converges to $g_0$ as $L'_1$ approaches negative infinity and $L'_2$ approaches positive infinity within the same space. Therefore,
$$\langle g_0, Tf_0 \rangle =\lim\limits_{\overset{L_1\to -\infty}{L_2\to +\infty}}
\lim\limits_{\overset{L'_1\to -\infty}{L'_2\to +\infty}}
\sum\limits_{k=L_1}^{L_2}\sum\limits_{k'=L'_1}^{L'_2}\Big\langle
D^{M}_{k'}g_1, D^*_{k'} T D_{k}D^{M}_{k}f_1\Big\rangle.$$
It is important to note that by Theorem \ref{th1.1}, $T D_{k},$ is a $G$-invariant H\"older function and hence $D^*_{k'} T D_{k}D^{M}_{k}f_1$ is a function in $L_G^2(\R^N).$ Estimating  $\sum\limits_{k=L_1}^{L_2}\sum\limits_{k'=L'_1}^{L'_2}\Big\langle
D^{M}_{k'}g_1, D^*_{k'} T D_{k}D^{M}_{k}f_1\Big\rangle
$ requires the subsequent lemma, whose proof we shall defer to the Appendix.

\begin{lemma}\label{le4.2}
    Let $T$ be a singular integral operator with a kernel $K$ defined in Definition \ref{def1.1} with $T(1)(x)=T^*(1)(x)=0$ and $T\in WBP_G$. Then
   $$\bigg|\int\limits_{\R^N}\int\limits_{\R^N} D_{k'}(x,u)K(u,v)D_k(v,y)
    dudv\bigg|\lesssim 2^{-|k-j|\varepsilon'}
    \frac{2^{(-j\vee -k)\varepsilon}}{\big(2^{(-j\vee -k)}+d(x,y)\big)^{N+\varepsilon}},$$
where $\varepsilon$ is the regularity exponent of $K$ and $a\vee b=\max\{a,b\}.$
\end{lemma}
Applying Lemma \ref{le4.2} and the Cotlar--Stein Lemma, we obtain
$$\sum\limits_{k=L_1}^{L_2}\sum\limits_{k'=L'_1}^{L'_2}\Big|\Big\langle
D^{M}_{k'}g_1, D^*_{k'} T D_{k}D^{M}_{k}f_1\Big\rangle\Big|
\lesssim
\|f_1\|_{L_G^2(\R^N)}\|g_1\|_{L_G^2(\R^N)}\lesssim
\|f_0\|_{L_G^2(\R^N)}\|g_0\|_{L_G^2(\R^N)}$$
for all integers $L_1, L_2, L'_1$ and $L'_2.$
It follows that
$$|\langle g_0, Tf_0 \rangle| =\Bigg|\lim\limits_{\overset{L_1\to -\infty}{L_2\to +\infty}}
\lim\limits_{\overset{L'_1\to -\infty}{L'_2\to +\infty}}
\sum\limits_{k=L_1}^{L_2}\sum\limits_{k'=L'_1}^{L'_2}\Big\langle
D^{M}_{k'}g_1, D^*_{k'} T D_{k}D^{M}_{k}f_1\Big\rangle\Bigg|
\lesssim \|f_0\|_{L^2(\R^N)}\|g_0\|_{L^2(\R^N)}.$$

Recall that the space $L^2_G(\mathbb{R}^N)$ consists of all functions $f$ in $L^2(\mathbb{R}^N)$ satisfying $f(x)=f(\sigma(x))$ for all $\sigma \in G$. To finish the proof, we also need the following result, whose proof is postponed to the Appendix.

\begin{lemma} \label{le4.3}
 For $0<\eta\le 1,$  the space $C_{G,0}^\eta(\R^N)$ is dense in $L^2_G(\R^N).$
\end{lemma}
Applying Lemma \ref{le4.3} gives the boundedness of $T$ in $L_G^2(\R^N).$
The proof of Theorem \ref{1.2} is thus concluded under the assumptions that $T(1)(x)= T^*(1)(x)=0$ and $T\in WBP_G$. To remove the assumptions $T(1)(x)=0$ and $ T^*(1)(x)=0,$ we need to introduce the following paraproduct operator.

\begin{definition}
Let $b\in BMO_G.$ The paraproduct operator associated with $b$ is defined by
$$\Pi_b(f)(x)=\sum_{k\in \Bbb Z} {\widetilde D}_k\{D_k(b)(\cdot)S_k(f)(\cdot)\}(x),$$
where $S_k, D_k$ and ${\widetilde D}_k$ are the same as in Theorem \ref{th2.2}.
\end{definition}

The properties of the paraproduct operator are given by the following lemma, whose proof is provided in the Appendix.
\begin{lemma} \label{le4.4}
The paraproduct operator $\Pi_b$ satisfies the following properties:
\begin{eqnarray*}
    &\textup{\bf (i)}& \Pi_b {\rm\ is\ a\ sigular \ integral\ operator \ defined\ in\ Definition~\ref{def1.1}\ and\ is \ bounded\ on }\ L_G^2(\R^N);\\
    &\textup{\bf (ii)}& \Pi_b(1)(x)=b(x);\\
    &\textup{\bf (iii)}& (\Pi_b)^*(1)=0.
    \end{eqnarray*}
\end{lemma}

We now define a modified operator ${\widetilde T}(f)$ by setting ${\widetilde T}(f)=T(f)-\Pi_{T1}(f)-(\Pi_{T^*1})^*(f)$. As in the classical case, by invoking Lemma \ref{le4.4}, we establish that ${\widetilde T}$ is a singular integral operator defined in Definition \ref{def1.1}, and it fulfills the conditions ${\widetilde T}(1)={\widetilde T}^*(1)=0$ and ${\widetilde T}\in WBP_G.$ Consequently, ${\widetilde T}$ is bounded on $L_G^2(\R^N)$. This, together with part {\bf(i)} of Lemma \ref{le4.4}, implies the boundedness of the original operator $T$ on $L_G^2(\R^N).$
The proof of Theorem \ref{th1.2} is now concluded.
\end{proof}

\section{Appendix}

To prove Lemma \ref{le4.1}, we need the following estimates for
$D_k$ and $D_k^{M}$.

\begin{lemma}\label{Lemma 3.1}   \sl
    Let $0<s<1$. Then {\bf (i)} $\|D_kf\|_{L_G^\infty} \lesssim 2^{-ks}\|f\|_{G,s}$; {\bf (ii)} $\|D_kf\|_{G,s} \lesssim 2^{ks}\|f\|_{L_G^\infty}$;
    {\bf (iii)}  $\|D_kf\|_{G,\beta}\lesssim 2^{k(\beta-s)}\|f\|_{G,s}$
    if $0<s\le\beta< 1$;
    {\bf (iv)} $\|D^{M}_kf\|_{G,s}\lesssim M\|f\|_{G,s}$.

\end{lemma}

\begin{proof}
    For {\bf (i)}, the cancellations of $D_k$ gives
    $$D_kf(x)=
    \int\limits_{\R^N}D_k(x,y)[f(y)-f(x)]dy.$$
    Since $D_k(x,y)=0$ for $d(x, y)\ge 2^{3-k}$,
    the size condition of $D_k$ and the smoothness condition of $f$ yield
    \begin{align*}
    |D_kf(x)|&\lesssim \|f\|_{G,s}\int\limits_{d(x, y)\le 2^{3-k}}2^{kN}d(x, y)^sdy
    \lesssim 2^{-ks}\|f\|_{G,s}.
    \end{align*}
    For {\bf (ii)}, the smoothness condition of $D_k$ gives
    \begin{align*}
    |D_kf(x)-D_kf(y)|
    &=\bigg|\int\limits_{\R^N} (D_k(x,z)-D_k(y,z))f(z)dz\bigg|\\
    &\le \|f\|_{L_G^\infty} \bigg(\int\limits_{d(x, z)\le 2^{3-k}}+\int\limits_{d(y, z)\le 2^{3-k}}\bigg)(2^kd(x, y))^s 2^{kN}dz,
    \end{align*}
    which implies $\|D_kf\|_{G,s} \lesssim 2^{ks}\|f\|_{L_G^\infty}.$

    To estimate {\bf (iii)},  if  $d(x, y)\le 2^{-k},$ using
    the the cancellations of $D_k$ and the smoothness condition of $f,$ we get
    \begin{align*}
    |D_kf(x)-D_kf(y)|
    &=\bigg|\int\limits_{\R^N} [D_k(x,z)-D_k(y,z)][f(z)-f(x)]dz\bigg|\\
    &\lesssim  \bigg(\int\limits_{d(x, z)\le 2^{3-k}}+\int\limits_{d(y, z)\le 2^{3-k}}\bigg)(2^kd(x, y))^\beta 2^{kN}d(x, z)^s\|f\|_{G,s}dz\\
    &\lesssim d(x, y)^\beta 2^{k(\beta-s)}\|f\|_{G,s}.
    \end{align*}

    When $d(x, y)> 2^{-k}$, {\bf (i)} gives $|D_kf(x)-D_kf(y)|\lesssim 2^{-ks}\|f\|_{G,s}
    \lesssim d(x, y)^\beta 2^{k(\beta-s)}\|f\|_{G,s}.$

The assertion {\bf (iii)} is a consequence of these estimates, while {\bf (iv)} follows from {\bf (iii)} with $\beta = s.$
This concludes the proof.
\end{proof}

We are now ready to give the proof of Lemma \ref{le4.1}.
\begin{proof}[Proof of Lemma \ref{le4.1}]
    We first show that $T_M(f)(x)=\sum\limits_{k=-\infty}^\infty D_kD^{M}_k(f)(x)$ is
    well defined on $\Lambda^s_G.$ To this end, let $f\in C^\eta_{G,0}$
    with $\eta>s$ and set $G_kf(x):=D_kD^{M}_kf(x)$. Observe that if $f\in C^\eta_{G,0}$
    then $f\in L_G^\infty$ and hence, $\|D^M_k(f)\|_{L_G^1(\R^N)}\le CM\|f\|_{L_G^\infty}.$
    By {\bf (iv)}, $\|D^M_k(f)\|_{G,\eta}\lesssim M\|f\|_{G,\eta}.$  Therefore
    \begin{align*}
    |G_k(f)(x)|&=\bigg|\int\limits_{\R^N} D_k(x,y)D^M_k(f)(y)dy\bigg|\lesssim 2^{kM}\|f\|_{L_G^\infty},
    \end{align*}
    and
    \begin{align*}
    |G_k(f)(x)|&=\Big|\int\limits_{\R^N} D_k(x,y)D^M_k(f)(y)dy\Big|=\Big|\int\limits_{\R^N} D_k(x,y)[D^M_k(f)(y)-D^M_k(f)(x)]dy\Big|\\
    &\lesssim \|f\|_{G,\eta}\int\limits_{\R^N} |D_k(x,y)|d(x, y)^\eta dy\lesssim 2^{-k\eta}\|f\|_{G,\eta}.
    \end{align*}
    These two estimates imply that if $f\in C^\eta_{G,0}$ then the series $\sum\limits_{k=-\infty}^\infty D_kD^{M}_k(f)(x)$ converges uniformly. Moreover, for given
    $x,y\in \R^N$, choose $k_0\in \Bbb Z$ such that
    $2^{-k_0}\le d(x, y)\le 2^{-k_0+1}$. Then by using
    Lemma \ref{Lemma 3.1}  and by splitting the sum over $k$ into the sums over $k\ge k_0$ and over $k\le k_0$, we obtain that
    \begin{equation}\label{eq 2.5}
    \bigg|\sum\limits_{k=-\infty}^\infty [G_kf(x)-G_kf(y)]\bigg|
    \lesssim 2^{-k_0s}\|f\|_{G,s}+2^{k_0(\beta-s)}d(x, y)^\beta\|f\|_{G,s}
    \lesssim d(x, y)^s\|f\|_{G,s}.
    \end{equation}
    Hence if $f\in C^\eta_{G,0}$ with $\eta>s$,
    then the series $\sum\limits_kD_kD^{M}_kf$ converges in $\Lambda_G^s$ norm. Observe that $C^\eta_{G,0}$ with $\eta > s$ is dense in $\Lambda_G^s.$ This implies that $T_M$ extends to $\Lambda_G^s.$ Indeed, if $f\in \Lambda_G^s,$ then there exists a sequence $f_n\in C^\eta_{G,0}, \eta>s,$ such that  $\|f_n-f\|_{G,s}$ tends to zero as $n$ tends to $\infty.$ Let $T_M(f)(x)=\lim\limits_{n\to \infty}T_M(f_n)(x).$ Then $T_M$ is bounded on $\Lambda_G^s$ and moreover, $\|T_M(f)\|_{G,s}\lesssim \|f\|_{G,s}$ for $f\in \Lambda_G^s.$

    To show that $\|{T}_{M}-I\|_{G,s}\to 0$ as $M\to +\infty,$ it suffices to prove that ${R}_{M} \to 0$ as $M\to +\infty$ with respect to the norm in $C^s_{G}$. To this end, we rewrite
    \begin{align*}
    {R}_{M}f
    &=\sum\limits_{k=-\infty}^\infty\sum\limits_{\{j\in \Bbb Z:\, |k-j|> {M}\}}{D}_k{D}_j
    =\sum\limits_{\{\ell \in \Bbb Z:\,|\ell|> M\}}\sum\limits_{k=-\infty}^\infty D_kD_{k+\ell} f \\
    &=\sum\limits_{k=-\infty}^\infty D_k(I-S_{k+M})f
    +\sum\limits_{k=-\infty}^\infty D_kS_{k-M-1}f.
    \end{align*}
    Since $\int\limits_{\R^N} S_k(x,y)dy=1$ for $k\in \Bbb Z$, we have
    \begin{align*}
    |(I-S_{k+M})f(x)|
    &
    =\bigg|\int\limits_{\R^N} S_{k+M}(x,y)[f(x)-f(y)]dy\bigg|
    \lesssim 2^{-(M+k)s}\|f\|_{G,s}
    \end{align*}
    and hence $\|(I-S_{k+M})f\|_{L_G^\infty}\lesssim 2^{-(M+k)s}\|f\|_{G,s}.$
    This estimate, together with Lemma \ref{Lemma 3.1} and an application  of the same proof for \eqref{eq 2.5} imply that
    $\big\|\sum\limits_{k=-\infty}^\infty D_k(I-S_{k+M})f\big\|_{G,s}\le 2^{-Ms}\|f\|_{G,s},$ which gives $\big\|\sum\limits_{k=-\infty}^\infty D_k(I-S_{k+M})\big\|_{\Lambda_G^s\mapsto
        \Lambda_G^s}\to 0$ as $M\to +\infty$.

    To estimate $\sum\limits_{k=-\infty}^\infty D_kS_{k-M-1}f,$ the second term of $R_Mf,$    let $H_kf=D_kS_{k-M-1}f$ and denote by $H_k(x,y)$ the kernel of $H_k$.  Then $\int\limits_{\Bbb R^N} H_k(x,y)dy=D_kS_{k-M-1}(1)=D_k(1)=0$
    and $H_k(x,y)=0$ if $d(x, y)\ge 2^{4-(k-M)}$.
    By the cancellation of $D_k$ and the smoothness of $S_{k-M-1}$,
    \begin{align*}
    |H_k(x,y)|
    &=\bigg|\int\limits_{\R^N}
    D_k(x,z)\big[S_{k-M-1}(z,y)-S_{k-M-1}(x,y)\big]dz\bigg|\\
    &\lesssim \int\limits_{B_d(x,2^{3-k})}2^{kN}\cdot 2^{k-M-1}d(x, z) \cdot 2^{(k-M-1)N}
    dz\lesssim 2^{-M}2^{(k-M-1)N}.
    \end{align*}
    Thus, for $f\in \Lambda_G^s,$
    \begin{align*}
    &|H_kf(x)|
    =\bigg|\int\limits_{\R^N}
    H_k(x,y)[f(y)-f(x)]dy\bigg|\\
    &\lesssim \int\limits_{d(x, y)\le 2^{4-(k-M)}}2^{-M}2^{(k-M-1)N}d(x, y)^s\|f\|_{G,s}dy\lesssim 2^{-M}2^{-(k-M)s}\|f\|_{G,s}.
    \end{align*}
    This implies that
    \begin{equation}\label{eq 2.6}
    \|H_kf\|_{ L_G^\infty}\lesssim 2^{-M}2^{-(k-M)s}\|f\|_{G,s}.
    \end{equation}
    If $d(x, x')\le 2^{4-(k-M)}$, then
    \begin{equation}\label{eq 2.7}
    \begin{aligned}
|H_k(x,y)-H_k(x',y)|
    &= \bigg|\int\limits_{\R^N} \big[D_k(x,z)-D_k(x',z)\big]S_{k-M-1}(z,y)dz\bigg| \\
    &\lesssim \int\limits_{\{d(x, z)\le 2^{2-k}\atop \text{or}\ d(x', z)\le 2^{2-k}\}}2^kd(x, x')2^{kN}2^{(k-M-1)N}dz \\
    &\lesssim 2^kd(x, x')2^{(k-M-1)N}.
    \end{aligned}
    \end{equation}
    For $d(x, y)\le 2^{4-(k-M)}$, applying \eqref{eq 2.7} yields
    \begin{align*}
    |H_kf(x)-H_kf(y)|
    &=\bigg|\int\limits_{\R^N}[H_k(x,z)-H_k(y,z)][f(z)-f(x)]dz \bigg|\\
    &\lesssim \int\limits_{\substack{d(x, z)\le 2^{4-(k-M)}\atop \text{ or}\ d(y, z)\le 2^{4-(k-M)}}}2^kd(x, y)2^{(k-M-1)N}d(x, z)^s\|f\|_{G,s} dz\\
    &\lesssim 2^k2^{-(k-M)s}d(x, y)\|f\|_{G,s}.
    \end{align*}
    For $d(x, y)> 2^{4-(k-M)}$, the estimate \eqref{eq 2.6} implies
    $$
    |H_kf(x)-H_kf(y)|\lesssim 2^{-M}2^{-(k-M)s}\|f\|_{G,s}
    \lesssim 2^k2^{-(k-M)s}d(x, y)\|f\|_{G,s}.
    $$
    These estimates imply that $H_k(f)(x)\in C_G^1$ with the norm bounded by
    \begin{equation}\label{eq 2.8}
    \|H_kf\|_{G,1}  \lesssim 2^k2^{-(k-M)s}\|f\|_{G,s}.
    \end{equation}
    Using the fact that $\|f\|_{G,\beta}\le \|f\|_{G,\infty}^{1-\beta}\|f\|_{G,1}^\beta, 0<\beta<1,$ the estimates \eqref{eq 2.6} and \eqref{eq 2.8} yield
    \begin{equation}\label{eq 2.9}
    \|H_kf\|_{G,\beta} \lesssim 2^{-M(1-2\beta)}2^{-(k-M)(s-\beta)}\|f\|_{G,s}.
    \end{equation}
    Given
    $x,y\in \R^N$, choose $k_1$ such that
    $2^{-k_1}\le d(x, y)\le 2^{-k_1+1}$.
    The estimates in \eqref{eq 2.6} and \eqref{eq 2.9} imply that, for $s<\beta,$
    \begin{align*}
    &\bigg|\sum\limits_{k=-\infty}^\infty [H_kf(x)-H_kf(y)]\bigg|\\
    &\le \sum\limits_{\{k:k\ge k_1\}} |H_kf(x)-H_kf(y)|
    +\sum\limits_{\{k:k< k_1\}}|H_kf(x)-H_kf(y)| \\
    &\lesssim  \sum\limits_{\{k:k\ge k_1\}}2\|H_kf\|_{ L_G^\infty}
    +\sum\limits_{\{k:k< k_1\}}d(x, y)^\beta \|H_kf\|_{G,\beta} \\
    &\lesssim  \sum\limits_{\{k:k\ge k_1\}} 2^{-M}2^{-(k-M)s}\|f\|_{G,s}
    + \sum\limits_{\{k:k< k_1\}}d(x, y)^\beta 2^{-M(1-2\beta)}2^{-(k-M)(s-\beta)}\|f\|_{G,s} \\
    &\lesssim 2^{-k_1s}2^{-M+Ms}\|f\|_{G,s}+ 2^{-M(1-2\beta)}2^{M(s-\beta)}2^{k_1(\beta-s)}d(x, y)^\beta\|f\|_{G,s}\\
    &\lesssim  2^{-M(1-2\beta)}\Big(2^{M(s-2\beta)}+2^{M(s-\beta)}\Big)d(x, y)^s\|f\|_{G,s}\\
    &\lesssim 2^{-M(1-2\beta)}d(x, y)^s\|f\|_{G,s}.
    \end{align*}
    Therefore, we have
    $$\Big\|\sum\limits_{k=-\infty}^\infty H_k f\Big\|_{G,s}\lesssim 2^{-M(1-2\beta)}\|f\|_{G,s}\qquad \text{for}\quad s<\beta<1.$$
    If $s<\frac12 $, we can choose $\beta$
    so that $2^{-M(1-2\beta)}\to 0$ as $M\to +\infty.$ The proof of Lemma \ref{le4.1} is concluded.
\end{proof}

Next, we present the
\begin{proof}[\bf {Proof of Lemma \ref{le4.2}}]

    We prove Lemma \ref{le4.2} in the crucial case when $k'\le k$ and $d(x,y)\le C2^{-k'}.$ The remaining case $d(x,y)>C2^{-k'}$ is similar but easier. Let $\eta_0\in C^\infty(\R^N)$ be 1 on the unit ball and 0 outside its double. Set $\eta_1=1-\eta_0.$ Then we have
    \begin{align*}
    \int\limits_{\R^N}&\int\limits_{\R^N} D_{k'}(x,u)K(u,v)D_k(v,y)
    dudv\\
    &=\int\limits_{\R^N} \big[D_{k'}(x,u)-D_{k'}(x,y)\big]K(u,v)D_k(v,y)
    dudv\\
    &=\int\limits_{\R^N}\int\limits_{\R^N}\eta_0(\frac{d(u,y)}{2^{-k}}) \big[D_{k'}(x,u)-D_{k'}(x,y)\big]K(u,v)D_k(v,y)
    dudv\\
    &\quad +\int\limits_{\R^N}\int\limits_{\R^N}\eta_1(\frac{d(u,y)}{2^{-k}}) \big[D_{k'}(x,u)-D_{k'}(x,y)\big][K(u,v)-K(u,y)]D_k(v,y)
    dudv\\
    &=:A+B.
    \end{align*}

Applying the weak boundedness property (Definition \ref{def1.2}) with  $f(v)=2^{-kN}D_k(v,y)$
and $g(u)=\big (\frac{2^{k}}{2^{k'}}\big)2^{-k'N}\eta_0(\frac{d(u,y)}{2^{-k}}) \big[D_{k'}(x,u)-D_{k'}(x,y)\big] $, we obtain
$$|A|= \left|\left\langle \big (\frac{2^{-k}}{2^{-k'}}\big)2^{k'N} g, T(2^{kN} f)\right\rangle \right|\lesssim 2^{-kN} \cdot \big (\frac{2^{-k}}{2^{-k'}}\big)2^{k'N}\cdot 2^{kN}\sim  2^{-|k-k'|}
    \frac{2^{(-k'\vee -k)}}{\big(2^{(-k'\vee -k)}+d(x,y)\big)^{N+1}},$$
    where the last inequality follows from the facts that $k'\le k$ and $d(x,y)\le C2^{-k'}.$

To estimate $B$, using the smoothness of $K(x,y)$ together with the estimate
$$|D_{k'}(x,u)-D_{k'}(x,y)|\lesssim
\frac{d(u,y)}{2^{-k'}+d(u,y)}2^{k'N},$$
we obtain that
\begin{align*}
|B|&\lesssim \int\int\limits_{d(u,y)\ge2^{-k}}\frac{d(u,y)}{2^{-k'}+d(u,y)} 2^{k'N}\frac{2^{-k}}{d(u,y)^{N+1}} |D_k(v,y)|dudv\\
    &\lesssim  2^{-k}2^{k'N}\int\limits_{d(u,y)\ge2^{-k'}}d(u,y)^{-N-1}du +\frac{2^{-k}}{2^{-k'}}2^{k'N}\int\limits_{2^{-k'}\ge d(u,y)\ge2^{-k}}d(u,y)^{-N}du\\
    &\lesssim  \big(1+ \log \frac{2^{-k'}}{2^{-k}}\big)\frac{2^{-k}}{2^{-k'}}2^{k'N}\\
    &\lesssim  2^{-|k-k'|\varepsilon}
    \frac{2^{(-k'\vee -k)}}{\big(2^{(-k'\vee -k)}+d(x,y)\big)^{N+1}},
\end{align*}
where the last inequality follows from the facts that $0<\varepsilon<1, k'\le k$ and $d(x,y)\le C2^{-k'}.$ This completes the proof of Lemma \ref{le4.2}.
\end{proof}

The proof of Lemma \ref{le4.3} is given as follows.
\begin{proof}[\bf {Proof of Lemma \ref{le4.3}}]
    Let $\psi$ be a function on $\mathbb{R}$ defined by
    $$
    \psi(x)=
    \begin{cases}
    e^{-\frac{1}{1-|x|^2}},& |x|<1,\\
    0,& |x|\ge 1.
    \end{cases}
    $$
    This function belongs to $C_0^\infty(\mathbb{R})$. Consider the function $c(x)$ given by
    $$
    c(x)=\int\limits_{\mathbb{R}^N}\psi(d(x,y))dy,
    $$
    which satisfies $c(x) \ge \int\limits_{\mathbb{R}^N}\psi(|x-y|)dy \gtrsim |B(x,\frac12)| \gtrsim 1$.

Let $\varphi(x,y)=\frac{1}{c(x)}\psi(d(x,y))$. Then $\varphi$ is supported in $\{ (x,y): d(x,y)\le 1\}$ and satisfies $\int\limits_\mathbb{R} \varphi(x,y)dy=1$. For any $\varepsilon>0$, let $\varphi_\varepsilon (x,y)=\frac{1}{\varepsilon^{N}}\varphi(\frac{x}{\varepsilon},\frac{y}{\varepsilon})$. Noting that
    $$ c\bigg(\frac{x}{\varepsilon}\bigg)=\int\limits_{\mathbb{R}^N}\psi\bigg(d\big(\frac{x}{\varepsilon},y\big)\bigg)dy
=\int\limits_{\mathbb{R}^N}\psi\bigg(d\big(\frac{x}{\varepsilon},\frac{y}{\varepsilon}\big)\bigg)\frac{dy}{\varepsilon}=
    \frac1{\varepsilon^{ N}}\int\limits_{\mathbb{R}^N}\psi\bigg(d\big(\frac{x}{\varepsilon},\frac{y}{\varepsilon}\big)\bigg)dy,
    $$
we have
\begin{equation}\label{n5.6}
    \int\limits_{\mathbb{R}^N}\varphi_\varepsilon (x,y)dy=\frac{1}{\varepsilon^{ N}}\int\limits_{\mathbb{R}^N} \frac{1}{c(\frac{x}{\varepsilon})}\psi\bigg(d\big(\frac{x}{\varepsilon},\frac{y}{\varepsilon}\big)\bigg)dy=1.
    \end{equation}

    For any $\delta>0$, and any function $f$ in $L^2_G(\mathbb{R}^N)$, using classical techniques from real analysis, we can find a $G$-invariant continuous function $g(x)$ with support contained in some large Euclidean ball $B(0,R)$ such that $\|f-g\|_2<\delta/2$. 
Clearly, $g$  is uniformly continuous on $\mathbb{R}^N$. Therefore, for any $0<\delta'<\delta\cdot(2\sqrt{|B(0,R+1)|})^{-1}$, there exists a constant $0<\eta<1$ such that $|g(x)-g(y)|<\delta'$ whenever $|x-y|<\eta$. Now for any $\varepsilon>0$, we define $h_\varepsilon (x)=\int\limits_{\mathbb{R}^N}\varphi_\varepsilon (x,y)g(y)dy$. Note that
    $$
    \|h_\varepsilon -g\|_2=\Big(\int\limits_{\mathbb{R}^N}\big|\int\limits_{\mathbb{R}^N}\varphi_\varepsilon (x,y)[g(y)-g(x)]dy\big|^2dx\Big)^\frac12.
    $$
Observe that for any point $(x,y)\in \R^{2N}$ such that $\varphi_\varepsilon (x,y)[g(y)-g(x)]\neq 0$, we have $d(x,y)<\varepsilon$ and $g(x)-g(y)\neq 0$. Choose $\sigma=\sigma_{x,y}\in G$ so that $|x-\sigma(y)|=d(x,y)<\varepsilon$; when $\varepsilon<\eta,$ using the reflection invariance property and uniform continuity of $g$, we conclude that
    $$
     |g(x)-g(y)|=|g(x)-g(\sigma(y))|<\delta'.$$
  Moreover, since $\supp g\subset B(0,R)$ and $g(x)-g(\sigma(y))\neq 0$, at least one of the two points $x$ and $\sigma(y)$ belongs to $B(0,R)$; in either case, we have $|x|<R+1$ (note that $|x-\sigma(y)|<1$).
These observations, together with \eqref{n5.6}, yield that
    $$
    \|h_\varepsilon -g\|_2 \le \delta' \Big(\int\limits_{|x|<R+1}dx\Big)^\frac12=\delta' |B(0, R+1)|^\frac12< \frac{\delta}{2}
    $$
whenever $0<\varepsilon<\eta$.
Consequently, $\|h_\varepsilon -f\|_2\le \|h_\varepsilon -g\|_2+\|g-f\|_2<\delta$ whenever $0<\varepsilon<\eta$, and hence $\lim\limits_{\varepsilon \to 0^+}\|h_\varepsilon -f\|_2=0$.

To finish the proof, it suffices to show that $h_\varepsilon \in C_{G,0}^\eta(\R^N).$ If $d(x,x')\le \varepsilon,$ then
    \begin{align*}
    &|h_\varepsilon(x)-h_\varepsilon(x')|\\
    &=\Big|\int\limits_{\R^N}[\varphi_\varepsilon (x,y)-\varphi_\varepsilon (x',y)]g(y)dy\Big|\\
    &=\Bigg|\int\limits_{\R^N}\frac1{\varepsilon^{ N}}\bigg[\frac1{c(\frac{x}{\varepsilon})}\psi\bigg(\frac{d(x,y)}{\varepsilon}\bigg)-\frac1{c(\frac{x'}{\varepsilon})}\psi\bigg(\frac{d(x',y)}{\varepsilon}\bigg)\bigg]g(y)dy\Bigg|\\
    &\lesssim_\varepsilon\Bigg|\int\limits_{\R^N}\bigg[\bigg(\frac1{c(\frac{x}{\varepsilon})}-\frac1{c(\frac{x'}{\varepsilon})}\bigg)\psi\bigg(\frac{d(x,y)}{\varepsilon}\bigg)+\frac1{c(\frac{x'}{\varepsilon})}\bigg(\psi\bigg(\frac{d(x,y)}{\varepsilon}\bigg)-\psi\bigg(\frac{d(x',y)}{\varepsilon}\bigg)\bigg)\bigg]g(y)dy\Bigg|.\\
    \end{align*}
  Note that $c(\frac{x}{\varepsilon})=\int\limits_{\R^N}\psi(d(\frac{x}{\varepsilon},y))dy\sim_{\varepsilon} \int\limits_{\R^N}\psi(\frac{d(x,y)}{\varepsilon}) dy\gtrsim |\mathcal{O}_B(x,\frac{\varepsilon}{2}) |\sim_\varepsilon 1$ and $c(\frac{x}{\varepsilon})\lesssim |\mathcal{O}_B(x,\varepsilon)|\sim_\varepsilon 1,$
    which implies  $c(\frac{x}{\varepsilon})\sim_\varepsilon 1.$
  Moreover,
    \begin{align*} \bigg|c\bigg(\frac{x}{\varepsilon}\bigg)-c\bigg(\frac{x'}{\varepsilon}\bigg)\bigg|&
=\bigg|\int\limits_{\R^N}\bigg[\psi\bigg(d(\frac{x}{\varepsilon},y)\bigg)-\psi\bigg(d(\frac{x'}{\varepsilon},y)\bigg)\bigg]dy\bigg|\\
    &\sim_{\varepsilon}\bigg|\int\limits_{\R^N}\bigg[\psi\bigg(\frac{d(x,y)}{\varepsilon}\bigg)-\psi\bigg(\frac{d(x',y)}{\varepsilon}\bigg)\bigg]dy\bigg|\\
    &\lesssim_\varepsilon d(x,x'),
    \end{align*}
which further implies
    \begin{align*}  \bigg|\frac1{c(\frac{x}{\varepsilon})}-\frac1{c(\frac{x'}{\varepsilon})}\bigg|=\frac{|c(\frac{x}{\varepsilon})-c(\frac{x'}{\varepsilon})|}{c(\frac{x}{\varepsilon})c(\frac{x'}{\varepsilon})}\lesssim_\varepsilon d(x,x').
    \end{align*}
 Hence, $|h_\varepsilon(x)-h_\varepsilon(x')|\lesssim_{\varepsilon,g} d(x,x').$
Since $h_\varepsilon$ have compact supports, we conclude that  $h_\varepsilon \in C_{G,0}^\eta(\R^N)$
for $0<\eta\le 1.$
    This completes the proof of Lemma \ref{le4.3}.
\end{proof}

Finally, let us give the
\begin{proof}[\bf {Proof of Lemma \ref{le4.4}}]
Observe that $\Pi_b(x,y)$, the kernel of $\Pi_b$, can be written as
    $$\Pi_b(x,y)=\sum\limits_{k\in \Bbb Z}\, \int\limits_{\R^N} {\widetilde D}_k(x,z)D_k(b)(z)S_k(z,y)dz$$
    and a simple calculation shows that $\Pi_b(x,y)$ satisfies all conditions in Definition \ref{def1.1} if $b\in {\rm BMO}_G(\R^N)$. We only show (i), the boundeness of $\Pi_b$ when $b\in {\rm BMO}_G(\R^N)$ on $L_G^2(\R^N)$, as (ii) and (iii) are easy to verify (see \cite{DH} for more details). To this end, we first claim that $\Pi_b(f)$ is $G$-invariant and bounded on $L_G^2(\R^N)$ for $b\in {\rm BMO}_G(\mathbb R^N)$ and $f\in L_G^2(\R^N)$. Indeed, it is easy to see that for $f\in L_G^2(\R^N)$ and each $k, {\widetilde D}_k\{D_k(b)(\cdot)S_k(f)(\cdot)\}(x)\in L_G^2(\R^N).$ Moreover, for every $f,g\in L_G^2(\R^N)$ and any fixed positive integer $M$, we have
    \begin{align}\label{Carleson}
    &|\langle \sum\limits_{|k|\leq M}{\widetilde D}_k\{ D_{k}(b)(\cdot)S_k(f)(\cdot)\}, g \rangle|=\Bigg|\int\limits_{\R^N}\sum\limits_{|k|\leq M}
    S_k(f)(y)D_{k}(b)(y){\widetilde D}_k(g)(y)dy\Bigg|\\
    &\le \Bigg(\int\limits_{\R^N}\sum\limits_{k}
    |S_k(f)(y)D_{k}(b)(y)|^2dy\Bigg)^{1\over2}\Bigg(\int\limits_{\R^N}\sum\limits_{k}
    |{\widetilde D}_k(g)(y)|^2dy\Bigg)^{1\over2}.\nonumber
    \end{align}
    It is straightforward  to see that the second factor on the right-hand side above is bounded by
    a constant multiple of $\|g\|_{ L_G^2(\R^N)}$, due to the $L_G^2$-boundedness of the Littlewood--Paley square function.

    Thus, it suffices to consider the first factor. Note that
    \begin{align*}
    \int\limits_{\R^N}\sum\limits_{k}
    |S_k(f)(y)D_{k}(b)(y)|^2dy= \int_0^\infty \sum\limits_{k}\int\limits_{\{y:|S_k(f)(y)|^2>\lambda\}}
    |D_{k}(b)(y)|^2dyd\lambda. \\
    \end{align*}
We denote $E_\lambda^*:=\{ x\in \mathbb R^N:  f^*(x)>\lambda \},$
    where $f^*(x):=\sup\limits_{(y,k):\ |x-y|<2^{-k}}|S_k(f)(y)|^2$. Performing the Whitney decomposition of $E_\lambda^*$, we have
    $E_\lambda^* = \bigcup\limits_{j=1}^\infty Q_j,$ where each $Q_j$ is a cube with center $x_{Q_j}$ and
 side length $l_{Q_j}.$
    We claim that
    $$S:=\sum\limits_{k}
    \int\limits_{\{y:|S_k(f)(y)|^2>\lambda\}}
    |D_{k}(b)(y)|^2dy\lesssim \|b\|_{\rm BMO_G}^2\cdot |E_\lambda^*|.
    $$
    Indeed, if $|S_k(f)(y)|^2>\lambda$, then
    $f^*(x)>\lambda$ for every $x$ satisfying $|x-y|<2^{-k},$ i.e., $B(y,2^{-k})\subseteq E_\lambda^* =\bigcup\limits_{j=1}^\infty Q_j.$
    Suppose $y\in Q_j$ for some $j,$ then by the property of the Whitney decomposition, we have $\ell(Q_j)\ge c2^{-k}$ for some constant $c>0.$
We write
    \begin{align*}
    S&\le \sum\limits_{j=1}^\infty  \sum\limits_{k:\ell(Q_j)\ge c2^{-k}}
    \int\limits_{\{y\in Q_j:|E_k(g)(y)|^2>\lambda\}}
    |D_{k}(b)(y)|^2dy\\
    &\le \sum\limits_{j=1}^\infty  \sum_{k:\ell(Q_j)\ge c2^{-k}}
    \int\limits_{\{y\in Q_j:|E_k(g)(y)|^2>\lambda\}}
    |D_{k}(b_1)(y)|^2dy\\
    &\qquad +\sum_{j=1}^\infty  \sum\limits_{k:l(Q_j)\ge c2^{-k}}
    \int\limits_{\{y\in Q_j:|E_k(g)(y)|^2>\lambda\}}
    |D_{k}(b_2)(y)|^2dy\\
    &=:Term_1+Term_2,
    \end{align*}
    where $b_1(x) := (b(x)-b_B)\chi_{{\mathcal O}(2\sqrt{N}B)}(x)$, $b_2(x) := (b(x)-b_B)\chi_{\mathbb R^N\backslash {\mathcal O}(2\sqrt{N}B)}(x)$ and $b_B=\frac1{|B|}\int\limits_B b(x)dx$ with $B=B(x_j,l(Q_j)),$ ${\mathcal O}(2\sqrt{N}B)=\{x: d(x,x_{Q_j})\le 2\sqrt{N}\ell(Q_j)\}.$
    Via the $L^2_G$-boundedness of the Littlewood--Paley square function, we know that
    \begin{align*}
    &Term_1\le\sum\limits_{j=1}^\infty  \sum\limits_{k\in \Z}
    \int\limits_{\R^N}
    |D_{k}(b_1)(y)|^2dy \le C\sum\limits_{j=1}^\infty\|b_1\|_{ L^2(\R^N) }^2.
    \end{align*}
    Noting that $b(x)=b(\sigma(x))$ for each $\sigma\in G$ we have
    \begin{align*}
    \|b_1\|_{ L^2(\R^N) }^2&=\int\limits_{x:d(x,x_j)\le 2\ell(Q_j)}|b(x)-b_B|^2dx\le \sum\limits_{\sigma\in G}\int\limits_{x:|\sigma(x)-x_j|\le 2\ell(Q_j)}|b(x)-b_B|^2dx\\
    &=|G|\cdot \int\limits_{x:|x-x_j|\le 2\ell(Q_j)}|b(x)-b_B|^2dx\lesssim \|b\|_{\rm BMO}^2 \cdot |Q_j|,
    \end{align*}
    which yields
    \begin{align*}
    &Term_1\lesssim \|b\|_{\rm BMO}^2\cdot |E_\lambda^*|.
    \end{align*}

    For $Term_2$, by using the size estimate of the kernel of $D_k$,
    we have
    \begin{align*}
    Term_2
    &\lesssim  \sum\limits_{j=1}^\infty  \sum\limits_{k:l(Q_j)\ge c2^{-k}}\int\limits_{Q_j}\bigg| \int\limits_{\mathbb R^N\backslash {\mathcal O}(2\sqrt{N}B)} {2^{-k\varepsilon} \over (2^{-k}+d(y,z))^{N+\varepsilon}} |b(z)-b_B|dz\bigg|^2dy\\
    &\lesssim \sum\limits_{j=1}^\infty  \sum\limits_{k:l(Q_j)\ge c2^{-k}}\int\limits_{Q_j}\bigg| \sum\limits_{u=0}^\infty\int\limits_{\{z\in \R^N:2^{u}\ell(Q_j)\le d(z,x_j)\le 2^{u+1}\ell(Q_j)\}} {2^{-k\varepsilon} \over (2^u\ell(Q_j))^{N+\varepsilon}} |b(z)-b_B|dz\bigg|^2dy.
    \end{align*}

    Noting that
    \begin{align*}
    &\int\limits_{\{z:d(z,x_j)\le 2^{u+1}\ell(Q_j)\}} |b(z)-b_B|dz\le\sum\limits_{\sigma\in G}\int\limits_{\{z:|\sigma(z)-x_j|\le 2^{u+1}\ell(Q_j)\}} |b(z)-b_B|dz\\
    &=|G|\cdot \int\limits_{\{z:|z-x_j|\le 2^{u+1}\ell(Q_j)\}} |b(z)-b_B|dz\\
    &\lesssim \int\limits_{2^{u+1}B} |b(z)-b_{2^{u+1}B}|dz+|2^{u+1}B|\sum\limits_{i=0}^{u}  |b_{2^{i+1}B}-b_{2^iB}|\\
    &\lesssim 2^{uN}(u+1)|B|\cdot \|b\|_{\rm BMO},
    \end{align*}
we derive
    \begin{align*}
    Term_2
    &\lesssim \|b\|_{\rm BMO}^2 \cdot\sum\limits_{j=1}^\infty  \sum\limits_{k:l(Q_j)\ge c2^{-k}}\int\limits_{Q_j}\bigg| \sum_{u=0}^\infty {2^{-k\varepsilon} \over (2^u\ell(Q_j))^{N+\varepsilon}} 2^{uN}(u+1)|B| \bigg|^2dy\\
    &\lesssim \|b\|_{\rm BMO}^2 \cdot\sum\limits_{j=1}^\infty  \sum\limits_{k:l(Q_j)\ge c2^{-k}}|Q_j|\bigg| {2^{-k\varepsilon} \over \ell(Q_j)^{\varepsilon}}\bigg|^2\\
    &\lesssim \|b\|_{\rm BMO}^2 \cdot\sum\limits_{j=1}^\infty  |Q_j|\lesssim \|b\|_{\rm BMO}^2\cdot |E_\lambda^*|.
    \end{align*}

    Combining the estimates for $Term_1$ and $Term_2$ above, we get
    $$\sum\limits_{k}
    \int\limits_{\{y:|S_k(f)(y)|^2>\lambda\}}
    |D_{k}(b)(y)|^2dy\lesssim \|b\|_{\rm BMO}^2\cdot |E_\lambda^*|,
    $$
    which implies
    \begin{align*}
    &&\int\limits_{\R^N}\sum\limits_{k}
    |S_k(f)(y)D_{k}(b)(y)|^2dy\lesssim \|b\|_{\rm BMO}^2 \int_0^\infty |E_\lambda^*| d\lambda
    =  \| b\|_{{\rm BMO}}^2\int\limits_{\R^N} f^*(x)dx
    \lesssim \| b\|_{{\rm BMO}}^2 \|f\|_{ L^2}^2.
    \end{align*}
This shows that $\Pi_b(f)$ is $G$-invariant and is bounded on $L_G^2(\R^N).$  The proof of  Lemma \ref{le4.4} is complete.
\end{proof}

\bigskip

\bigskip

\noindent {\bf Acknowledgement}: Li is supported by ARC DP 220100285.  Tan is supported by National Natural Science Foundation of China (No. 12071272). Wu is supported by National Natural Science Foundation  of China (No. 12071473).

\bigskip
\bigskip

\medskip
\vskip 0.5cm
\medskip
\vskip 0.5cm

\noindent Department of Mathematics, Auburn University, AL
36849-5310, USA.

\noindent {\it E-mail address}: \texttt{hanyong@auburn.edu}

\medskip
\vskip 0.5cm
\noindent Department of Mathematics, Macquarie University, NSW, 2109, Australia.

\noindent {\it E-mail address}: \texttt{ji.li@mq.edu.au}

\medskip
\vskip 0.5cm

\noindent  Department of Mathematics, Shantou University, Shantou,
515063, R. China.

\noindent {\it E-mail address}: \texttt  {cqtan@stu.edu.cn }

\smallskip
\vskip 0.5cm

\noindent  Department of Mathematics, School of Science, Westlake University, Hangzhou, Zhejiang 310030, P. R. China.

\noindent {\it E-mail address}: \texttt{wangzipeng@westlake.edu.cn }

\smallskip
\vskip 0.5cm

\noindent  Department of Mathematics, China University of Mining \& Technology,
Beijing 100083, China.

\noindent {\it E-mail address}: \texttt{wuxf@cumtb.edu.cn}

\end{document}